\documentclass[11pt,reqno]{amsart}
\usepackage[foot]{amsaddr}
\usepackage{amsthm,amssymb,amsfonts}
\usepackage{amsmath}
\usepackage{amscd} 
\usepackage{setup}
\usepackage{microtype}

\usepackage[normalem]{ulem}
\usepackage[all]{xy}
\usepackage{manfnt}

\usepackage{cases}

\newcommand{\N}{\mathbb{N}}
\renewcommand{\P}{\mathbb{P}}

\newcommand{\R}{\mathbb{R}}

\usepackage{bbm}
\renewcommand{\1}{\mathbf{1}}

\newcommand{\B}{\mathcal{B}}

\newcommand{\F}{\mathcal{F}}
\newcommand{\G}{\mathcal{G}}

\newcommand{\X}{\mathcal{X}}

\newcommand{\ceq}{\coloneqq} 

\newcommand{\e}{\varepsilon}

\DeclarePairedDelimiter\normd{\lVert}{\rVert} 

\newcommand{\bpar}[1]{\left(#1\right)}

\newcommand{\ind}[1]{\1_{[#1]}} 

\newcommand{\E}[1]{{\mathbf E}\left[#1\right]}				

\newcommand{\p}[1]{{\mathbf P}\left\{#1\right\}}

\newcommand{\eqdist}{\ensuremath{\stackrel{\mathrm{d}}{=}}}

\usepackage{subfig}
\usepackage{adjustbox}

\newcommand{\bnum}{b_n}

\newcommand{\tpath}[2]{\xymatrix@1{{#1}\ar[r]|{\G}&{#2}} } 
\newcommand{\tpathn}[2]{\xymatrix@1{{#1}\ar[r]|{\G_n}&{#2}} }

\def\tagform@#1{\maketag@@@{[\ignorespaces#1\unskip\@@italiccorr]}}

\makeatletter
\newcommand{\leqnomode}{\tagsleft@true\let\veqno\@@leqno}
\newcommand{\reqnomode}{\tagsleft@false}
\makeatother

\title{Temporal connectivity of Random Geometric Graphs}

\author[A. Brandenberger]{Anna Brandenberger$^*$}
\address{$^*$Department of Mathematics, MIT 
\texttt{abrande@mit.edu}}

\author[S. Donderwinkel]{Serte Donderwinkel$^{\dagger}$}
\address{$^{\dagger}$Bernoulli Institute and CogniGron, University of Groningen
\texttt{s.a.donderwinkel@rug.nl}
}

\author[C. Kerriou]{C\'eline Kerriou$^\ddagger$}
\address{$^\ddagger$Department of Mathematics and Computer Science, Universit{\"a}t zu K{\"o}ln
\texttt{ckerriou@math. uni-koeln.de}
}

\author[G. Lugosi]{G\'abor Lugosi$^\mathsection$}
\address{$^\mathsection$
ICREA, Pg. Lluís Companys 23, 08010 Barcelona, Spain; 
Department of Economics and Business, Pompeu Fabra University; Barcelona School of Economics
\texttt{gabor.lugosi@gmail.com}
}

\author[R. Mitchell]{Rivka Mitchell$^\circ$}
\address{$^\circ$Department of Mathematics, University of Oxford
\texttt{rivka.mitchell@maths.ox.ac.uk}
}

\begin{document}

\begin{abstract}
A temporal random geometric graph is a random geometric graph in which all edges are endowed with a uniformly random time-stamp, representing the time of interaction between vertices.
In such graphs, paths with increasing time stamps indicate
the propagation of information. We determine a threshold for the existence of monotone increasing paths between all pairs of vertices in temporal random geometric graphs. 
{The results reveal that temporal connectivity appears at a significantly larger edge density than simple connectivity of the underlying random geometric graph. This is in contrast with Erd\H{o}s-R\'enyi random graphs in which the thresholds for temporal connectivity and simple connectivity are of the same order of magnitude.}
Our results hold for a family of ``soft" random geometric graphs as well as the standard random geometric graph.  
\end{abstract}

\maketitle 

\begin{figure}[hbtp]
    \centering
    \includegraphics[width=.49\textwidth]{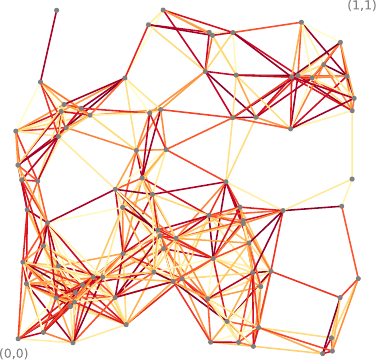}
    \includegraphics[width=.49\textwidth]{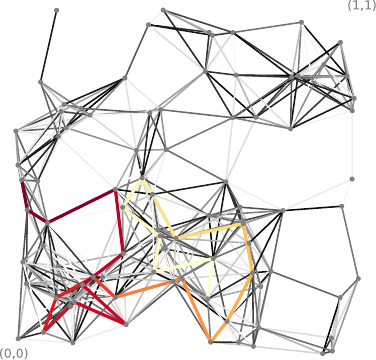}
    \caption{A temporal random geometric graph on $[0,1]^2$ where edges that are coloured darker arrive later (left), and its longest monotone increasing path (right).}
    \label{fig:intro}
\end{figure}

\newpage 

\section{Introduction}

A \emph{temporal graph} is an edge-labeled graph $\G = (G, \sigma)$ where the underlying graph $G = (V(G),E(G))$ is a finite simple graph and $\sigma: E(G)\rightarrow \{1,\dots, |E(G)|\}$ is an ordering of the edges. In this paper we consider random temporal graphs in which $\sigma(E(G))$ is a uniform random permutation. Equivalently, one can generate a random temporal graph using independent, uniform, labels: each edge $e\in E(G)$ is assigned a label $\tau_e\sim \mathrm{Unif}[0,1]$.  For $e\in E(G)$ we say that $\tau_e$ is the \emph{time-stamp} of $e$, and for $f\in E(G)\setminus \{e\}$, we say that $e$ \emph{precedes} $f$ if $\tau_e \le \tau_f$. We use this construction in the sequel, and denote the resulting graph as $\G = (G, (\tau_e)_{e\in E(G)})$

The study of such graphs is largely motivated by the modeling of dynamic networks, see Tang et al.\ \cite{Tang2013ApplicationsOT}. A temporal graph represents a network where interactions occur at particular times. In temporal networks one may study whether it is possible to transmit information, or an infectious disease, from one individual to another, as a chain of transmissions can only occur along interactions that are increasing in time.

Let $\G = (G, (\tau_e)_{e\in E(G)})$ be a temporal graph. For vertices $u,v\in V(G)$ we say that  there is a \emph{temporal path} from $u$ to $v$, denoted $\tpath{u}{v}$, if there exists a path from $u$ to $v$ comprised of non-decreasing edge time-stamps. Further, we let $\ell(\tpath{u}{v})$ denote the length of the shortest temporal path from $u$ to $v$, that is, the minimal number of edges on any monotone increasing path from $u$ to $v$. 
It is readily seen that the existence of a temporal path $\tpath{u}{v}$ does not imply the existence of a temporal path $\tpath{v}{u}$, and furthermore the existence of temporal paths $\tpath{u}{v}$ and $\tpath{v}{w}$ does not imply the existence of a temporal path $\tpath{u}{w}$. Following the terminology of \cite{becker2023giant}, we say that a vertex $u\in \G$ is a \emph{temporal source} if  there exist temporal paths $\tpath{u}{v}$ for all $v\in\G$.
Moreover, we call $\G$ \emph{temporally connected}
if every vertex of $\G$ is a temporal source.

The temporal random graph model can be seen as a random version of a well-known combinatorial problem posed by Chv\'atal and Koml\'os~\cite{chvatal1971some}, 
which asks for the minimal value of the length of the longest monotone path, when considering the edge orderings of the complete graph on $n$ vertices, henceforth denoted by $K_n$. This problem is well studied, with lower and upper bounds remaining far apart until recent years, see, e.g.,\ Graham and Kleitman \cite{graham1973increasing},
Calderbank, Chung, and Sturtevant \cite{calderbank1984increasing}, Bucić et al. \cite{bucic2020nearly}. The average case with base graph $G = K_n$, was first considered by Lavrov and Loh \cite{lavrov2016increasing}, and Martinsson \cite{martinsson2019most} later proved that with high probability, there exists a monotone path of length $n-1$. When $G$ is an Erd\H{o}s--R\'enyi random graph, the resulting temporal graph is called the \emph{random simple temporal graph}, first studied by Angel, Ferber, Sudakov and Tassion~\cite{Angel_Ferber_Sudakov_Tassion_2020}. 

Temporal connectivity of random simple temporal graphs   was studied in detail by Casteigts, Raskin, Renken and Zamarev \cite{casteigts2021} and Becker et al.\ \cite{becker2023giant}.
In particular, it is shown in 
\cite{becker2023giant} that the threshold for temporal connectivity of random simple temporal graphs is $\sim 3\log n/n$, which is just a factor of $3$ larger than the connectivity threshold of the underlying Erd\H{o}s-R\'enyi random graph.
More recent works have studied the shortest and longest increasing paths between typical vertices
(Broutin, Kamčev, and Lugosi~\cite{NicNinaGabor}) and the size of the largest temporal clique (Atamanchuk, Devroye, and Lugosi~\cite{atamanchuk2024size,becker2023giant}). 

In this paper, we initiate the study of \emph{temporal random geometric graphs}. 
In this model the vertices of the underlying graph $G_n$ correspond to randomly drawn points in a Euclidean space. Such graphs may be better suited for modeling epidemiological processes than random simple temporal graphs, as the geometry allows for spatial closeness (or closeness encoding similarity) of individuals to affect interaction probabilities.
Our main results establish thresholds for temporal 
connectivity in certain temporal random geometric graphs.

Let $n \in \N$, $K: \R^+ \to [0,1]$, $r_n \in \R_+$, and let $\normd{\cdot}$ denote the Euclidean distance on {the torus} $[0,1]^d$. A \textit{soft random geometric graph} $G_n = (\X_n, K, r_n)$ in dimension $d \geq 2$ is a graph on $n$ uniform random points $\X_n$ in the unit torus $[0,1]^d$, such that for every pair of vertices $u,v\in \X_n$, an edge exists with probability
\[ \p{ (u,v) \in E(G_n)\mid \X_n} = K\bpar{\frac{\normd{u-v}}{r_n} }.\] 
The edges are conditionally independent given $\X_n$. When $K(x) \ceq \ind{x \leq 1}$, $G_n$ is called a \textit{hard random geometric graph}. A \emph{temporal random geometric  graph} with $n$ vertices is a temporal graph $\G_n = (G_n, (\tau_e)_{e\in E(G_n)})$ where $G_n = (\mathcal{X}_n, K, r_n)$ is a soft random geometric graph on $n$ vertices. Our main result shows that under regularity conditions on $K$, with high probability, the graph becomes temporally connected when the radius $r_n$ exceeds a constant multiple of $n^{-1/(d+1)}$, while the graph is temporally disconnected if $r_n$ is smaller than another constant times $n^{-1/(d+1)}$.

\begin{thm}\label{thm:main_temp} Let $\G_n = (G_n, (\tau_e)_{e \in E(G_n)})$ be a temporal random geometric graph with $G_n = (\mathcal{X}_n, K,r_n)$ satisfying the following: 
\makeatletter \def\tagform@#1{\maketag@@@{[\ignorespaces#1\unskip\@@italiccorr]}}\makeatother
\leqnomode
\begin{align}
\tag{\color{blue}\bf A1}\label{a1}
&\text{There exists } \alpha > 0 \text{ such that } K(x) \geq \alpha \text{ for all } x \leq 1. \\ 
\tag{{\color{blue}\bf A2}}\label{a2}
&\text{For } x > 1 \text{, }K(x) \leq \beta x^{-d}e^{- (x+1) \log (x+1)} \text{ for some $\beta > 0$}.
\end{align}
\reqnomode
\makeatletter
\def\tagform@#1{\maketag@@@{(\ignorespaces#1\unskip\@@italiccorr)}}
\makeatother
    Then, there exist constants $c_d, C_d> 0$  such that
\begin{enumerate}[label=(\roman*)]
      \item \label{thm:main_upper} if $r_n \leq c_d n^{-1/(d+1)}$ for all $n$ sufficiently large, then $\G_n$ is a.a.s.\ temporally disconnected, 
    \item  \label{thm:main_lower} if $r_n \geq C_d n^{-1/(d+1)}$ for all $n$ sufficiently large, then $\G_n$ is a.a.s.\ temporally connected. 
\end{enumerate}
\end{thm}
Observe that by monotonicity of temporal connectivity, Theorem \ref{thm:main_temp} implies that \ref{thm:main_upper} holds under just assumption [\ref{a2}] and \ref{thm:main_lower} holds under just assumption [\ref{a1}].

Note that the choice of $K(x) = \ind{x \leq 1}$ satisfies the regularity conditions [\ref{a1}] and [\ref{a2}], and therefore as a corollary of Theorem~\ref{thm:main_temp} we obtain a threshold for a.a.s.\ temporal connectivity of temporal hard random geometric graphs.
The assumptions [\ref{a1}] and [\ref{a2}] ensure that there are sufficiently many vertices at graph distance $O(r_n)$ to any vertex, and that there are not too many long edges in the soft random geometric graph.

Observe that the order of magnitude $n^{-1/(d+1)}$ of the critical radius is significantly larger than that of the critical radius for the connectivity of the underlying random geometric graph. Indeed, as it is well known from the theory of random geometric graphs, random geometric graphs become connected when the average degree of the graph is of the order $\log n$ (see \cite{Penrose2013}). On the other hand, our results show that temporal connectivity only occurs when the average degree becomes of the order of $n^{1/(d+1)}$. This is in stark contrast with Erd\H{o}s-R\'enyi random graphs in which temporal connectivity and simple connectivity both occur when the average degree is of the order of $\log n$.

Theorem \ref{thm:main_temp} establishes the order of magnitude of the temporal connectivity threshold. It is natural to conjecture that, similarly to the case of Erd\H{o}s-R\'enyi graphs, there is a sharp threshold for temporal connectivity, that is, there exists a constant $\kappa_d$ (depending on the kernel $K$) such that, for all $\varepsilon >0$, if $r_n \le (\kappa_d-\varepsilon)n^{-1/(d+1)}$, then 
$\G_n$ is a.a.s.\ temporally disconnected, while for 
$r_n \ge (\kappa_d+\varepsilon)n^{-1/(d+1)}$,
$\G_n$ is a.a.s.\ temporally connected. Our techniques are not sufficiently strong to establish such results and we leave the problem of existence of a sharp threshold and the value of $\kappa_d$ for challenging future research.

Becker et al.~\cite{becker2023giant} introduce various definitions of temporal connectivity. In particular, for the case of random simple temporal graphs, they establish sharp thresholds for the probability of the existence of a temporal path between two typical vertices and for the probability of a typical vertex being a temporal source. It is an interesting question to study these probabilities, along with the probability of temporal connectivity. As it is apparent from our proofs, the critical radius for  all these quantities is of the same order of magnitude, but we have no further information about their relationship.

In Section~\ref{sec:new_upper_bound} we prove the upper bound of Theorem~\ref{thm:main_temp}, that is, we determine a threshold for when $\G_n$ is a.a.s.\ temporally disconnected. In Section~\ref{sec:lower_bound} we prove the lower bound of Theorem~\ref{thm:main_temp} by mapping the construction of a temporal path to a directed percolation model.

\section{Upper bound: $\G_n$ temporally disconnected}\label{sec:new_upper_bound}
 Let $\mathcal{G}_n = (G_n, (\tau_e)_{e \in E(G_n)})$ be a temporal random geometric graph with $G_n = (\mathcal{X}_n, K, r_n)$ satisfying assumptions [\ref{a1}] and [\ref{a2}]. In this section, we will prove that there exists  $c>0$ such that if $r_n\leq cn^{-1/(d+1)}$ for all $n$ sufficiently large, then a.a.s.\ $\G_n$
has a vertex that is not a temporal source (and in particular $\G_n$ is not temporally connected). In particular, we establish the following proposition. 
 \begin{prop}
     Suppose that $\G_n = (G_n, (\tau_e)_{e\in E(G_n)})$ is a temporal random geometric graph with $G_n = (\mathcal{X}_n, K, r_n)$ satisfying assumptions \textup{[\ref{a1}]} and \textup{[\ref{a2}]}. Then there exists $c > 0$ such that if $r_n \le cn^{-1/(d+1)}$ for all $n$ sufficiently large, then
 \begin{equation}\label{eq:upbd_one_to_all_temp_conn}
     \p{\exists u \in \X_n \text{ such that } \forall v \in \X_n,\, \tpathn{u}{v}} = o(1).
 \end{equation}
 \end{prop}
 \begin{proof} 
    By \cite[Theorem 2.3]{Penrose2013}, there is a constant $\gamma_d>0$ such that if $r_n < (\log n/(\gamma_dn))^{1/d}$ then $G_n$ is disconnected with high probability. Therefore, we may assume that for all $n$ sufficiently large, $r_n \geq (\log n/(\gamma_dn))^{1/d}$. 
 
The event of having a temporal source is contained in the union of the following three events: all points $\X_n$ are contained in a ball of radius at most $\sqrt{d}(1-r_n)/2$ {around some point in $\X_n$}, there exists a long temporal path of length at least $1/3r_n$, or there exists a short temporal path between two points at distance at least  $\sqrt{d}(1-r_n)/2$. More precisely,
\begin{equation}\label{eq:terms_to_bound}
\begin{aligned}
    &\p{\exists u\in \X_n \text{ such that } \forall v\in \X_n, \tpathn{u}{v}}\\
    &\le \p{\exists u\in \X_n \text{ such that } \forall v\in \X_n, \|u-v\| < \frac{\sqrt{d}}{2}(1-r_n)}\\
    &\quad + \p{\exists u,v \in \X_n \text{ such that } \tpathn{u}{v}, \ell(\tpathn{u}{v}) \ge {\frac{1}{3r_n}}}\\
    &\quad+ \p{\exists u,v \in \X_n \text{ such that } \|u - v\| \ge \frac{\sqrt{d}}{2}(1-r_n),  \tpathn{u}{v}, \ell(\tpathn{u}{v}) < {\frac{1}{3r_n}}}.
\end{aligned}
\end{equation}
Thus, it suffices to show that each of the terms on the right-hand side of  \eqref{eq:terms_to_bound} are $o(1).$

The first term is $o(1)$ since we assume that $r_n \geq (\log n/(\gamma_dn))^{1/d}$ and so, for some constant $\gamma_d'>0$,
\begin{align}\label{eq:first_term}
     \p{\exists u \in \X_n\text{ such that } \forall v \in \X_n,\, \|u - v\| < \frac{\sqrt{d}}{2}(1-r_n)}
    & \leq n\left(\text{Vol}\left(\B(0,\frac{\sqrt{d}}{2}(1 - r_n)\right)\right)^{n-1} 
    \nonumber\\
    &\leq n\gamma_d'^{\,n-1}\left(\sqrt{d}e^{-r_n}/2\right)^{d(n-1)}\nonumber\\
    &=o(1).
\end{align}

To bound the second term, we will apply the first moment method. To this end, without loss of generality assume that $1/{3}r_n\in \N$, let $k = 1/{3}r_n$ and let $N_n^{(k)}$ denote the number of temporal paths in $\G_n$ with at least $k-1$ edges. Note that by choosing $k$ ordered vertices, and ensuring that the edges between them exist and have monotone increasing labels, it follows that 
\begin{align*}
    \E{N_n^{(k)}} & \leq   \binom{n}{k}k! \frac{1}{{(k-1)!}} \left(\alpha \gamma_d' r_n^d + \int_{ [0,1]^d \setminus \B(0,r_{{n}})} K(\|y\|/r_n)dy\right)^{k-1}.
\end{align*}
By our conditions on $K$, the integral in the above inequality is at most a constant multiple (depending on $d$) of $r_n^d$. Together with the bound $\binom{n}{k}\le \left(\frac{en}{k}\right)^k$, it follows that there exists $\gamma_d''> 0$ such that,
\begin{align}
      \E{N_n^{(k)}} & \leq (\gamma_d'')^{k-1} \left(\frac{en}{k}\right)^k k r_n^{d(k-1)}\nonumber \\
       &= (\gamma_d'')^{k-1} \exp\left(k\log n + k - (k-1)\log k + d(k-1)\log(r_n)\right)  \nonumber\\ 
    & = (\gamma_d'')^{k-1} \exp\left(k\log n + k - (k-1)\log k - d(k-1)\log k \right) \nonumber\\
    &\le (\gamma_d'')^{k-1} \exp\left(k\log n + k\left(1-(d+1)\frac{k-1}{k}\log k \right)\right) \label{eq:upper_bound_ENk}.
\end{align}  
If for some
$c>0$
it holds that
$r_n \leq cn^{-1/(d+1)}$ for $n$ sufficiently large, we have that $\log k = \log (1/{3}r_n) \ge \log(1/{3}c) + \log(n)/(d+1)$ and so \eqref{eq:upper_bound_ENk} is at most 
   \begin{equation}
    (\gamma_d'')^{k-1}\exp\left( \log n + \left(1-(d+1)\log\left(\frac{1}{{3}c}\right)\right)k + (d+1)\log\left(\frac{1}{{3}c}\right)\right),\label{eq:Xk-expectation}
\end{equation}
which is $o(1)$ {for sufficiently small $c > 0$}. Then 
\begin{align}\label{eq:second_term}
    \p{\exists u,v \in \X_n \text{ such that }\tpathn{u}{v}, \ell(\tpathn{u}{v}) \ge \frac{1}{{3}r_n}} = o(1),
    \end{align}
    by the first moment method.

    From \eqref{eq:first_term} and \eqref{eq:second_term}, we have that the first two terms in \eqref{eq:terms_to_bound} are $o(1)$.
    Thus, it remains to prove that connectivity via short paths is unlikely, i.e., to show that
\begin{align}\label{eq:prob_path_1}
     \p{\exists u,v \in \X_n \text{ : } \|u - v\| \geq \frac{\sqrt{d}}{2}(1-r_n), \tpathn{u}{v},\ell( \tpathn{u}{v}) < \frac{1}{{3}r_n} } = o(1).
\end{align}
First, note that we can work under the event that for every edge $uv \in E(G_n)$, $\normd{u-v} \leq r_n \log n$, since the probability that there exists an edge of length greater than $r_n \log n $ is, by a union bound, at most
\[
n^2 \gamma_d \int_{\log n}^\infty K(x) x^{d-1} dx = O\left( n^2 e^{-\log n \log\log n}\right) = o(1).
\]
For sufficiently large $n$, we have $\sqrt{d} (1-r_n) /2 \geq 1/3$ so we can then bound the left-hand side of \eqref{eq:prob_path_1} from above by the probability that there exists a temporal path $\tpathn{u}{v}$ with {$M \in ((3r_n \log n)^{-1}, (3r_n)^{-1})$} edges, such that the sum of its edge lengths is at least $1/3$.

We now compute the expected number of paths in $G_n$ with $M$ edges and such that the sum of the edge lengths is at least $1/3$. Consider such a path, and write its edge lengths $d_1, \dots, d_M$ with $\sum_{i=1}^M d_i \geq 1/3$. 
These edge lengths correspond to a vector $(\ell_i)_{i\in[M]} \in \N^{\geq 1}$ where $\ell_i \ceq \lceil d_i r_n^{-1} \rceil$, such that $L \ceq \sum_{i=1}^M \ell_i$ satisfies $L \geq r_n^{-1}/3$. For a given $L$, the number of such vectors is bounded from above by 
\[\binom{L-1}{M-1} = O\left( \bpar{\frac{eL}{M}}^M \right) = O\left( e^{M(\log(L/M) + 1)}\right).\]
Fixing $(\ell_i)_{i\in M}$, the probability that a path with $M$ edges and edge lengths $d_1,\ldots,d_M$, satisfying $(\ell_i - 1)r_n \leq d_i \leq \ell_i r_n$, exist can be bounded above by 
\begin{align*}
    \prod_{i=1}^M  \gamma_d \int_{(\ell_i-1) r_n}^{\ell_ir_n}  K(\xi/r_n)\xi^{d-1} d\xi 
    &= (\gamma_d r_n^d)^M \prod_{i=1}^M \Big( \int_{\ell_i-1}^{\ell_i} K(x)x^{d-1} dx \Big) \\ 
    &= O\left( (\gamma_d r_n^d)^M \prod_{i=1}^M e^{-\ell_i \log(\ell_i)}\right) \\ 
    &= O\left((\gamma_d r_n^d)^M \exp\left( -L \sum_{i=1}^M \frac{\ell_i}{L} \log(\ell_i /L) - \sum_{i=1}^M \ell_i\log(L) \right)\right) \\  
    &= O\left((\gamma_d r_n^d)^M \exp(-L \log (L/M))\right),
\end{align*}
where in the last equality we use the fact that $- \sum_{i=1}^M\tfrac{\ell_i}{L}\log(\ell_i/L) \leq \log M$. Therefore the expected number of paths in $G_n$ with $M$ edges and such that the sum of its edge lengths is at least $1/3$, is 
\[O\left((\gamma_dr_n^d)^Me^{M+M\log(L/M) - L\log(L/M)}\right).\]
Applying the first moment method we obtain that the probability that there exists a temporal path $\tpathn{u}{v}$ comprised of $M$ edges, such that the sum of the edge lengths is $1/3$, is 
\[O\left({n \choose M+1}(M+1)!\frac{1}{M!}(\gamma_dr_n^d)^Me^{M+M\log(L/M) - L\log(L/M)}\right),\]
where we have used the fact that such paths contain $M+1$ ordered vertices between which the edges exist and are ordered with monotone increasing edge labels. 
Since $L \ge r_n^{-1}/3 > M$, this is equal to 
\[O\left({n\choose M+1}(M+1)r_n^{dM}e^M\right),\]
which is $o(1)$ when $M \in ((r_n\log n)^{-1}, r_n^{-1}/3)$ by the same computation as \eqref{eq:upper_bound_ENk}. 
\end{proof}

\section{Lower bound: $\G_n$ temporally connected}\label{sec:lower_bound}

 In this section we show that there exists a constant $C_d > 0$ such that if $r_n \ge C_d n^{-1/(d+1)}$ for all $n$ sufficiently large, then $\G_n$ is a.a.s\ temporally connected. 
 
 By monotonicity, to prove Theorem~\ref{thm:main_temp} \ref{thm:main_lower}, it suffices to prove the statement for a hard random geometric graph where each edge is retained with probability $\alpha\in (0,1]$. That is for $\G_n = (G_n, (\tau_e)_{e\in E(G_n)})$, with $G_n = (\X_n, K,r_n)$ such that $K(x) := \alpha \ind{x\leq 1}$. More specifically, we prove the following proposition.

 \begin{prop}\label{prop:lower_bound}
     Fix $\alpha\in (0,1]$. For all $n \ge 1$ let $\G_n = (G_n, (\tau_e)_{e\in E(G_n)})$ be a temporal random geometric graph with $G_n = (\X_n, K, r_n)$ such that $K(x) = \alpha\ind{x\le 1}$ for all $x \ge 0$. Then there exists $C_d > 0$ such that if for all $n$  sufficiently large, $r_n \ge C_dn^{-1/(d+1)}$, then
     \[\p{\forall u,v\in \X_n,~ \tpathn{u}{v}} = 1-o(1).\]
 \end{prop}

 We present the proof in the specific case where $d = 2$, and discuss the extension to general dimensions in subsection \ref{sec:dim}. In this case, we henceforth simplify notation and let $C = C_2 > 0.$  We also assume that for all $n$  sufficiently large, $r_n = Cn^{-1/3}$ (since the case $r_n \ge Cn^{-1/3}$ follows by monotonicity).

 To prove Proposition \ref{prop:lower_bound}, we construct a temporal path between any two vertices $u,v\in \X_n$. First, we present the construction for the two furthest possible separated points on the torus $[0,1]^2$, namely $u = (0,0)$ and $v = (1/2,1/2).$ An analogous construction then holds for any pair of points, and the result follows by a union bound over all $u,v \in \X_n$.

To construct a temporal path from $(0,0)$ to $(1/2,1/2)$, we split $[0,1/2]^2$ into three regions, see Figure \ref{fig:regions} for an illustration. Fix $\varepsilon \in (0, 1/3)$. In the first region (the red region in Figure \ref{fig:regions}), we construct many monotone increasing paths from $(0,0)$ and ending at a point close to the diagonal $x+y = (4\sqrt{2})^{-1}n^\varepsilon r_n$. The monotone paths in this region are such that all edges have labels that are at most $1/4$, and the label of the $m$-th edge in each path falls in an interval $T_m \subset [0,1]$ with  $|T_m| \approx n^{-\varepsilon}$. The size of each interval $T_m$  is generous, and this generosity is such that if continued for too long, we would run out of admissible labels before reaching $(1/2,1/2)$. Therefore, in the second region, (the white region in Figure \ref{fig:regions}), we build onto the paths constructed in the first region, if possible, along edges whose labels are more restricted, that is, ranging from $1/4$ to $3/4$, and falling in intervals of length $\approx n^{-1/3}$. This construction is such that with high probability, at least one of the paths from the first region can be extended to have an endpoint close to the diagonal $x+y = 1 - (4\sqrt{2})^{-1} r_n(1 + n^\varepsilon/4).$ In the third region, (the yellow region in Figure \ref{fig:regions}), we proceed similarly to in the first region, being generous once again with the edge labels. By a symmetric argument, the path which reached the diagonal $x+y = 1 - (4\sqrt{2})^{-1}r_n(1 + n^\varepsilon/4)$ can be extended to a path with endpoint at the point $(1/2,1/2)$. 
Since the number of vertices {of} $\X_n$ lying in a given region of $[0,1]^2$ with area $((4\sqrt{2})^{-1}r_n)^2$ is concentrated around its mean, $N = (C^2/32)n^{1/3}$, to formalize this argument, it is convenient to frame the construction in terms of an inhomogeneous directed percolation model on the dual grid lattice where the directed edges in the percolation model are open with probabilities given by (\ref{def:perc_prob}).

    \subsection{From temporal graph to directed percolation on the dual grid lattice}\label{sec:directed_percolation}  Without loss of generality, assume that $((4\sqrt{2})^{-1}r_n)^{-1}, (2(4\sqrt{2})^{-1}r_n)^{-1} \in \N$. Partition $[0,1/2]^2$ into boxes of side-length $\ell_n = (4\sqrt{2})^{-1}r_n.$ Let $\bnum = (2\ell_n)^{-1}-1$ and for $i,j\in [\bnum]_0{\ceq\{0,1,\dots,\bnum\}}$, let 
    \begin{align}\label{def:boxes}
        \square_{i,j} := [i\ell_n, (i+1)\ell_n]\times [j\ell_n, (j+1)\ell_n],
    \end{align}
 so that by our choice for $\ell_n$, for $v\in \X_n\cap \square_{i,j}$ and $u\in \X_n\cap \square_{i',j'}$ with \[(i',j')\in \{(i+1,j), (i-1,j), (i,j+1), (i, j-1)\}\] it holds that $\p{(u,v)\in E(G_n)}=\alpha$.

In this section, we map the construction of a temporal path between the points $(0,0)$ and $(1/2,1/2)$ to a directed percolation on the dual grid lattice $\mathbb{L}_n = (\mathbb{V}_n,\mathbb{E}_n)$ with vertex set  
\[\mathbb{V}_n = \{\square_{i,j}~:~i,j\in [\bnum]_0\},\]
and directed edge set $\mathbb{E}_n$ being the union of the sets 

\[  \left\{(\square_{\bnum,j}, \square_{\bnum, j+1})~:~j\in [\bnum - 1]_0\right\},\] \[\left\{(\square_{i,\bnum}, \square_{i+1, \bnum})~:~i\in [\bnum-1]_0\right\},\]
and
\[\left\{(\square_{i,j}, \square_{i+1,j}), (\square_{i,j}, \square_{i,j+1})~:~ i,j\in [\bnum-1]_0\right\}. \]
See Figure \ref{fig:grid_lattice} for an illustration of $\mathbb{L}_n$.

\begin{SCfigure}[1][hbtp]
    \centering
    \tikzset{every picture/.style={line width=0.75pt}} 

\begin{tikzpicture}[x=0.75pt,y=0.75pt,yscale=-1,xscale=1]
    \def\num{10}
    
    \foreach \x in {40,60,...,120} {
        \draw[gray, opacity=0.5, -] (\x,20) -- (\x,120);
        \foreach \z in {30, 50,70,90} {
        \draw[color={rgb, 255:red, 74; green, 144; blue, 226 }, <-] (\x + 10, \z + 3) -- (\x+10, \z+17);
        \node at (\x + 10,\z + 20)[circle,fill,inner sep=1.5pt,color={rgb, 255:red, 74; green, 144; blue, 226 }]{};
        }
        \node at (\x + 10,10 + 20)[color={rgb, 255:red, 74; green, 144; blue, 226 },circle,fill,inner sep=1.5pt]{};

    }
    
    \foreach \y in {20,40,...,100} {
        \draw[gray, opacity=0.5] (40,\y) -- (140,\y);
        \foreach \z in {50,70,..., 110} {
        \draw[color={rgb, 255:red, 74; green, 144; blue, 226 }, ->] (\z+3, \y + 10) -- (\z + 17, \y+10);
        }
    }

    \fill[gray, opacity = 0.4] (100,100) rectangle (120,120);
   
    \draw[black, opacity=0.5] (40,20) -- (140,20) -- (140,120) -- (40,120) -- cycle;

\draw (39.83,127.85) node  [font=\small,color={rgb, 255:red, 0; green, 0; blue, 0 }  ,opacity=1 ] [align=left] {\begin{minipage}[lt]{24.25pt}\setlength\topsep{0pt}
$\displaystyle ( 0,0)$
\end{minipage}};
\draw (140.17,10.85) node  [font=\small,color={rgb, 255:red, 0; green, 0; blue, 0 }  ,opacity=1 ] [align=left] {\begin{minipage}[lt]{25.61pt}\setlength\topsep{0pt}
$\displaystyle ( 1/2,1/2)$
\end{minipage}};
\end{tikzpicture}
    \caption{Directed grid lattice with $\ell_n = 1/10$ with vertices and directed edges in blue. The filled square highlights vertex $\square_{0,3}$. \\
    }
    \label{fig:grid_lattice}
\end{SCfigure}
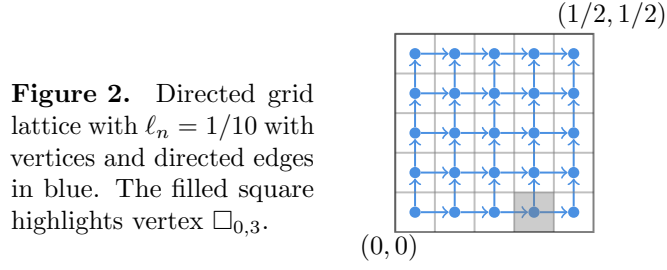
The percolation on the probability space $(\Omega, \mathbb{P}, \F)$, has configuration space $\Omega = \{0,1\}^{\mathbb{E}_n}$ where for $\omega\in \Omega$, and $\vec{e}\in \mathbb{E}_n$, $\omega(\vec{e}) = 1$ denotes the directed edge $\vec{e}$ being open. For any $\omega\in \Omega$, we say that the vertex $\square_{i,j}\in \mathbb{V}_n$ is connected to vertex $\square_{k,l}$, denoted $\square_{i,j}\rightarrow \square_{k,l}$, if there exists an open path in $\mathbb{E}_n$ from $\square_{i,j}$ to $\square_{k,l}$. That is, if there exists a sequence of open directed edges  $\vec{e}_0 = (\square_{i,j}, \square_{i_0,j_0})$, $\vec{e}_1 = ( \square_{i_0,j_0}, \square_{i_1,j_1}),\dots,$ $\vec{e}_m = ( \square_{i_m-1,j_m-1}, \square_{k,l})\in \mathbb{E}_n$. 

Informally, the idea of the proof of Proposition \ref{prop:lower_bound} is to map the construction of a temporal random geometric graph $\G_n = (G_n, (\tau_e)_{e\in E(G_n)})$ to a directed percolation of $\mathbb{L}_n$ such that an edge $(\square_{i,j}, \square_{i',j'})\in \mathbb{E}_n$ is open if in $\G_n$, for all vertices of $u\in \X_n\cap\square_{i,j}$ there exists a vertex $u\in\X_n\cap\square_{i',j'}$ such that the time-stamp $\tau_{uv}$ falls in a certain range $T_{i+j} \subset [0,1]$. Fix $\varepsilon \in (0,1/3)$. We define the ranges $(T_{i+j})_{i,j}$ as follows. For $c,d \in \mathbb{R}$, $(a,b]\subset \mathbb{R}$, we write $d+c(a,b]$ for $(d+ca, d+cb]$. Then for all $m \in [\bnum]_0$, let 
\begin{equation}\label{eq:ranges}T_{m} := \begin{cases}
    \left(\frac{m}{n^\varepsilon}, \frac{m+1}{n^{\varepsilon}}\right] & m\in R_1,\\
    \frac{1}{4} - \frac{(n^\varepsilon - 2)}{8\bnum - n^\varepsilon} + \frac{4(n^\varepsilon - 2)n^{1/3 - \varepsilon}}{8\bnum - n^\varepsilon} \left(\frac{m}{n^{1/3}}, \frac{m+1}{n^{1/3}}\right] & m \in R_2,\\
    1 - \frac{2\bnum}{n^{\e}} + \left(\frac{m}{n^\varepsilon}, \frac{m+1}{n^\varepsilon}\right] &m \in R_3,
\end{cases}\end{equation}
where $R_1 = [0, n^{\e}/4),\,  R_2 = [n^{\e}/4, 2\bnum -n^{\e}/4) $ and $R_3 = [2\bnum -n^{\e}/4, 2\bnum]$. See Figure \ref{fig:regions} for an illustration.

\begin{SCfigure}[.6][hbtp]
\scalebox{0.75}{
\tikzset{every picture/.style={line width=0.75pt}} 

\begin{tikzpicture}[x=0.75pt,y=0.75pt,yscale=-1,xscale=1]

\draw  [draw opacity=0] (40,20) -- (340,20) -- (340,320) -- (40,320) -- cycle ; \draw  [color={rgb, 255:red, 155; green, 155; blue, 155 }  ,draw opacity=0.5 ] (40,20) -- (40,320)(60,20) -- (60,320)(80,20) -- (80,320)(100,20) -- (100,320)(120,20) -- (120,320)(140,20) -- (140,320)(160,20) -- (160,320)(180,20) -- (180,320)(200,20) -- (200,320)(220,20) -- (220,320)(240,20) -- (240,320)(260,20) -- (260,320)(280,20) -- (280,320)(300,20) -- (300,320)(320,20) -- (320,320) ; \draw  [color={rgb, 255:red, 155; green, 155; blue, 155 }  ,draw opacity=0.5 ] (40,20) -- (340,20)(40,40) -- (340,40)(40,60) -- (340,60)(40,80) -- (340,80)(40,100) -- (340,100)(40,120) -- (340,120)(40,140) -- (340,140)(40,160) -- (340,160)(40,180) -- (340,180)(40,200) -- (340,200)(40,220) -- (340,220)(40,240) -- (340,240)(40,260) -- (340,260)(40,280) -- (340,280)(40,300) -- (340,300) ; \draw  [color={rgb, 255:red, 155; green, 155; blue, 155 }  ,draw opacity=0.5 ]  ;
\draw   (40,20) -- (340,20) -- (340,320) -- (40,320) -- cycle ;

\fill[color=red, opacity = 0.4] (40,220) rectangle (60,320) ;
\fill[color=red, opacity = 0.4] (60,240) rectangle (80,320) ;
\fill[color=White, opacity = 0.4] (60,240) rectangle (80,20);
\fill[color=White, opacity = 0.4] (40,220) rectangle (60,20);
\fill[color=White, opacity = 0.4] (80,260) rectangle (100,20);
\fill[color=White, opacity = 0.4] (100,280) rectangle (120,20);
\fill[color=White, opacity = 0.4] (120,300) rectangle (140,20);
\fill[color=White, opacity = 0.4] (140,320) rectangle (160,20);
\fill[color=White, opacity = 0.4] (160,320) rectangle (220,20);
\fill[color=White, opacity = 0.4] (220,320) rectangle (240,40);
\fill[color=White, opacity = 0.4] (240,320) rectangle (260,60);
\fill[color=White, opacity = 0.4] (260,320) rectangle (280,80);
\fill[color=White, opacity = 0.4] (280,320) rectangle (300,100);
\fill[color=White, opacity = 0.4] (300,320) rectangle (320,120);
\fill[color=White, opacity = 0.4] (320,320) rectangle (340,140);
\fill[color=yellow, opacity = 0.4] (320,20) rectangle (340,140);
\fill[color=yellow, opacity = 0.4] (300,20) rectangle (320,120);
\fill[color=yellow, opacity = 0.4] (280,20) rectangle (300,100);
\fill[color=yellow, opacity = 0.4] (260,20) rectangle (280,80);
\fill[color=yellow, opacity = 0.4] (240,20) rectangle (260,60);
\fill[color=yellow, opacity = 0.4] (220,20) rectangle (240,40);
\fill[color=red, opacity = 0.4] (80,260) rectangle (100,320) ;
\fill[color=red, opacity = 0.4] (100,280) rectangle (120,320) ;
\fill[color=red, opacity = 0.4] (120,300) rectangle (140,320) ;
\fill[color=red, opacity = 0.4] (140,320) rectangle (160,320) ;

\draw (90,300) node  [font=\small,color={rgb, 255:red, 0; green, 0; blue, 0 }  ,opacity=1 ] {\begin{minipage}[lt]{50pt}\setlength\topsep{0pt}
$i+j\in R_1$
\end{minipage}};
\draw (200,170) node  [font=\small,color={rgb, 255:red, 0; green, 0; blue, 0 }  ,opacity=1 ] {\begin{minipage}[lt]{50pt}\setlength\topsep{0pt}
$i+j\in R_2$
\end{minipage}};

\draw (300,40) node  [font=\small,color={rgb, 255:red, 0; green, 0; blue, 0 }  ,opacity=1 ] {\begin{minipage}[lt]{50pt}\setlength\topsep{0pt}
$i+j\in R_3$
\end{minipage}};
\draw (39.83,327.85) node  [font=\small,color={rgb, 255:red, 0; green, 0; blue, 0 }  ,opacity=1 ] [align=left] {\begin{minipage}[lt]{24.25pt}\setlength\topsep{0pt}
$\displaystyle ( 0,0)$
\end{minipage}};
\draw (340.17,10.85) node  [font=\small,color={rgb, 255:red, 0; green, 0; blue, 0 }  ,opacity=1 ] [align=left] {\begin{minipage}[lt]{25.61pt}\setlength\topsep{0pt}
$\displaystyle ( 1/2,1/2)$
\end{minipage}};

\draw (125,329) node  [font=\small,color={rgb, 255:red, 0; green, 0; blue, 0 }  ,opacity=1 ] [align=left] {\begin{minipage}[lt]{24.25pt}\setlength\topsep{0pt}
$\displaystyle (n^\varepsilon \ell_n/4,0)$
\end{minipage}};
\draw (250,10) node  [font=\small,color={rgb, 255:red, 0; green, 0; blue, 0 }  ,opacity=1 ] [align=right] {\begin{minipage}[lt]{140pt}\setlength\topsep{0pt}
$\displaystyle (1/2 - \ell_n(1+n^\varepsilon/4), 1/2)$
\end{minipage}};

\foreach \x in {40,60,80,100,120,140,160,180,200,220,240,260,280,300,320}{
\foreach \y in {20,40,60,80,100,120,140,160,180,200,220,240,260,280,300}
\fill[gray, opacity = 0.2] (\x + 10, \y+10) circle (2pt);
}
\end{tikzpicture}
}
\caption{The intervals $R_1, R_2, R_3$ should be thought of as indicating ``regions" in the grid in the sense that if, for example, $i+j\in R_1$, then $\square_{i,j}$ is a vertex in the portion of the grid which is highlighted in red above (and similarly $\square_{i,j}$ such that $i+j\in R_2$ is in white, and $\square_{i,j}$ for $i+j\in R_3$ in yellow).\\
} \label{fig:regions}
\end{SCfigure}
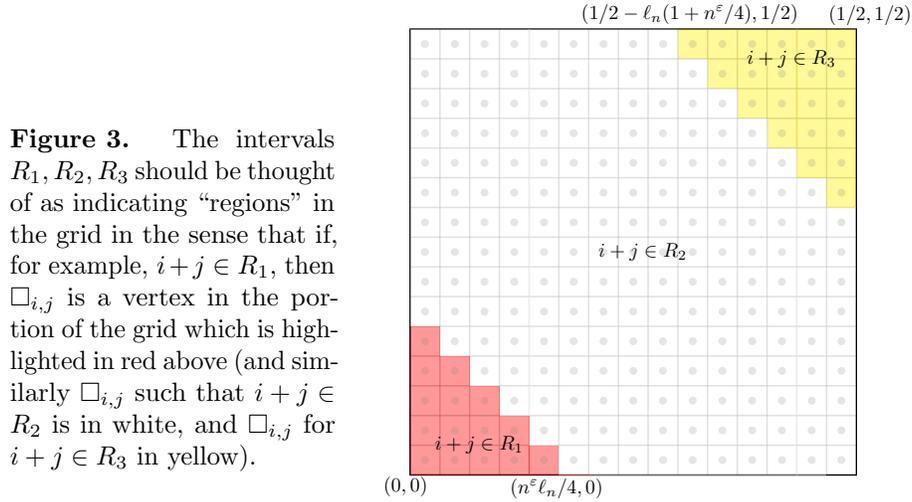

\begin{rem}\label{rem:sizes}
    The sizes of the regions in (\ref{eq:ranges}) are chosen specifically to work well with the number of edges in $\G_n$ that have one endpoint in $\square_{i,j}$ and another in $\square_{i+1,j}\cup \square_{i,j+1}$. More specifically, given $\X_n$, the number of edges between a fixed vertex in $\X_n\cap \square_{i,j}$ and vertices in $\X_n\cap \square_{i+1,j}$ with time-stamp in $T_{i,j}$ has distribution $\mathrm{Binomial}(|\X_n \cap \square_{i+1,j}|, |T_{i+j}|).$   Since $|\X_n \cap \square_{i+1,j}| \eqdist \mathrm{Binomial}(n,\mathrm{vol}(\square_{i+1,j}))$, the number of points in $\X_n\cap\square_{i+1,j}$ is highly concentrated around its mean, $n\cdot\mathrm{vol}(\square_{i+1,j})$. This implies that if $i+j\in R_2$, then a constant proportion of the edges between a fixed vertex in $\X_n\cap \square_{i,j}$ and vertices in $\X_n\cap (\square_{i+1,j}\cup \square_{i,j+1})$ have time-stamp in $T_{i+j}$. On the other hand, if $i+j\in R_1 \cup R_3$, then the proportion of such edges with time-stamps in $T_{i+j}$ is of the order of $n^{1/3 - \varepsilon}$. 
\end{rem}

Let $(B_{\vec{e},k})_{\vec{e}\in \mathbb{E}_n, k \in[(1+t)N]}$ be independent random variables satisfying 
\begin{align}
        B_{\vec{e}, k}  \stackrel{d}{=} \text{Binomial}((1-t)N, \alpha|T_{i+j}|).
\end{align}
We consider the directed percolation model with edges $\vec{e} = (\square_{i,j}, \square_{i',j'})\in \mathbb{E}_n$ open independently with probability $p_{\vec{e}}$, where
\begin{align}\label{def:perc_prob}
    p_{\vec{e}} := \begin{cases}
    \p{B_{\vec{e},1} \geq 1 }^{(1+t)N} & \text{ if } i+ j \in R_1 \cup R_3,   \\
         \p{B_{\vec{e},1} \geq 1 } & \text{ if } i+j \in R_2.
    \end{cases}
\end{align}
{We will denote the law of the directed percolation model by $\mathbb{P}$.}
\begin{lem}\label{lem:perc_to_diag}
    Let $\e \in (0,1/3)$ with $n^{\e}/4\in \mathbb{N}$. The probability of having an open path from $\square_{0,0}$ to every box $\square_{i,j}$ with $i+j = n^{\e}/4$ in our directed percolation model is bounded from below by
    \begin{align*}
        \P\left\{\bigcap_{i+j = \frac{n^\varepsilon}{4}}\left\{\square_{0,0} \to \square_{i,j}\right\} \right\} \geq 1 - \exp\left(-\Omega(n^{1/3 -\e})\right).
    \end{align*}
\end{lem}
\begin{proof}
    For all $\vec{e}\in \mathbb{E}_n$ such that $\vec{e} = (\square_{i,j}, \square_{i',j'})$ with $i+j < n^\varepsilon/4$, we have that $|T_{i+j}|=n^{-\e}$ and so $(B_{\vec{e},k})_{k\in [(1+t)N]}$ are i.i.d.\ random variables with distribution $\mathrm{Binomial}((1-t)N, \alpha n^{-\varepsilon}).$
    By Chernoff's inequality, see e.g.\,\cite{janson2000random}, the union bound  and \eqref{def:perc_prob}, 
    \begin{equation}\label{eq:lower_perc_prob}
        p_{\vec{e}} \ge 1 - (1+t)N\exp\left(-\frac{\alpha}{4}(1-t)Nn^{-\varepsilon}\right) = 1 - \exp(-\Omega(n^{1/3 - \varepsilon})).
    \end{equation}
    Each edge $\vec{e}\in \mathbb{E}_n$ is open independently with probability $p_{\vec{e}}$, and the events $\{\square_{0,0}\rightarrow \square_{i,j}\}_{i,j}$ are increasing in $\Omega = \{0,1\}^{\mathbb{E}_n}$. Therefore by the FKG inequality, see e.g.~\cite[Theorem 2.16]{boucheron2003concentration},
    \begin{align*}
        \P\left\{\bigcap_{i+j = \frac{n^\varepsilon}{4}}\left\{\square_{0,0} \to \square_{i,j}\right\} \right\} &\ge \prod_{i+j = \frac{n^\varepsilon}{4}}\mathbb{P}\{\square_{0,0}\rightarrow \square_{i,j}\}\\
        &\ge \prod_{i+j = \frac{n^\varepsilon}{4}}
        p_{\vec{e}}^{n^\varepsilon/4},
    \end{align*}
    where for the second inequality we have used the fact that $\mathbb{P}\{\square_{0,0}\rightarrow \square_{i,j}\}$ is at least the probability that a single path of length $n^\varepsilon/4-1$ is open. Combining this with (\ref{eq:lower_perc_prob}) and the inequality $(1-x)^k\ge 1-xk$ for $x\in (0,1)$ yields the result.
    \end{proof}
\begin{lem}\label{lem:perc_btw_diag}
     Let $\e\in (0,1/3)$ with $n^{\e}/4\in \mathbb{N}$, then the probability of having an open path from a box $\square_{i,j}$ with $i+j = n^{\e}/4$ to a box $\square_{k,l}$ with $k+l = 2\bnum-n^{\e}/4$ in our directed percolation model can be bounded from below by
     \begin{equation*}
        \P\left\{\bigcup_{i+j = \frac{n^{\varepsilon}}{4}} \;\bigcup_{k+l = 2\bnum - \frac{n^\varepsilon}{4}} \{\square_{i,j}\rightarrow \square_{k,l}\}\right\}
        \geq 1 - \exp\left(-\Omega(n^{\e} \wedge n^{1/3-\e})\right).
    \end{equation*}
\end{lem}
 
\begin{proof}
    We condition on the event \[\mathcal{E}_n = \bigcap_{i+j = \frac{n^\varepsilon}{4}} \{\square_{0,0} \rightarrow \square_{i,j}\},\] which by Lemma \ref{lem:perc_to_diag} occurs with probability $1 - \exp(-\Omega(n^{1/3 - \varepsilon})).$ 

    On this event, if there is no open path between $\{\square_{i,j}: i+j = n^\varepsilon/4\}$ and $\{\square_{k,l}: k+l = 2\bnum - n^\varepsilon/4\}$, then the connected component containing $\square_{0,0}$ and the vertices $\{\square_{i,j}: i+j = n^\varepsilon/4\}$ has a boundary which can be represented as a path in the lattice (rather than in the dual grid lattice) with mesh size $\ell_n$, with one endpoint in the top left-hand corner region of the unit square,  
    \[\mathrm{TC}_n:= \{(0,j\ell_n):j\in \{n^\varepsilon/4,\dots, \bnum + 1\}\}\cup \{(i\ell_n,0):i\in \{0,\dots,\bnum-n^\varepsilon/4\}\}\] and another in bottom right-hand corner of the unit square, \[\mathrm{BC}_n:=\{(0,j\ell_n):j\in \{0,\dots,\bnum-n^\varepsilon/4\}\}\cup \{(i\ell_n,0):i\in \{n^\varepsilon/4,\dots,\bnum+1 \}\}.\]
    See Figure \ref{fig:bound} for an illustration.
    \begin{SCfigure}[.8]
        \tikzset{every picture/.style={line width=0.75pt}} 

\begin{tikzpicture}[x=0.75pt,y=0.75pt,yscale=-1,xscale=1]
    \def\num{10}
    
    \foreach \x in {40,60,...,200} {
        \draw[gray, opacity=0.5, -] (\x,20) -- (\x,200);
        \foreach \z in {30, 50,70,90, 110, 130, 150, 170} {
        \node at (\x + 10,\z + 20)[circle,fill,inner sep=1.5pt,color={rgb, 255:red, 74; green, 144; blue, 226 }]{};
        }
        \node at (\x + 10,10 + 20)[color={rgb, 255:red, 74; green, 144; blue, 226 },circle,fill,inner sep=1.5pt]{};

    }

    \foreach \y in {20,40,...,180} {
        \draw[gray, opacity=0.5] (40,\y) -- (220,\y);
    }

    \draw[color={rgb, 255:red, 74; green, 144; blue, 226 }, ->] (50+3, 180 + 10) -- (50 + 17, 180+10);
    \draw[color={rgb, 255:red, 74; green, 144; blue, 226 }, ->] (70+3, 180 + 10) -- (70 + 17, 180+10);
     \draw[color={rgb, 255:red, 74; green, 144; blue, 226 }, ->] (90+3, 180 + 10) -- (90 + 17, 180+10);
      \draw[color={rgb, 255:red, 74; green, 144; blue, 226 }, ->] (50+3, 160 + 10) -- (50 + 17, 160+10);
    \draw[color={rgb, 255:red, 74; green, 144; blue, 226 }, ->] (70+3, 160 + 10) -- (70 + 17, 160+10);
    \draw[color={rgb, 255:red, 74; green, 144; blue, 226 }, ->] (50+3, 140 + 10) -- (50 + 17, 140+10);
   \draw[color={rgb, 255:red, 74; green, 144; blue, 226 }, ->] (50, 180 + 10) -- (50, 160+13);
    \draw[color={rgb, 255:red, 74; green, 144; blue, 226 }, ->] (70, 180 + 10) -- (70, 160+13);
   \draw[color={rgb, 255:red, 74; green, 144; blue, 226 }, ->] (90, 180 + 10) -- (90, 160+13);
   \draw[color={rgb, 255:red, 74; green, 144; blue, 226 }, ->] (50, 160 + 10) -- (50, 140+13);
   \draw[color={rgb, 255:red, 74; green, 144; blue, 226 }, ->] (50, 140 + 10) -- (50, 120+13);
   \draw[color={rgb, 255:red, 74; green, 144; blue, 226 }, ->] (70, 160 + 10) -- (70, 140+13);
   \draw[color={rgb, 255:red, 74; green, 144; blue, 226 }, ->] (90, 160 + 10) -- (90, 140+13);
       \draw[color={rgb, 255:red, 74; green, 144; blue, 226 }, ->] (90+3, 160 + 10) -- (107, 160+10);
               \draw[color=WildStrawberry, opacity = 0.75, ->] (150+3, 160 + 10) -- (150 + 17, 160+10);
       \draw[color=WildStrawberry, opacity = 0.75, ->] (150+3, 140 + 10) -- (150 + 17, 140+10);
       \draw[color=WildStrawberry, opacity = 0.75, ->] (150+3, 180 + 10) -- (150 + 17, 180+10);

        \draw[color={rgb, 255:red, 74; green, 144; blue, 226 }, ->] (130+3, 180 + 10) -- (130 + 17, 180+10);
        \draw[color={rgb, 255:red, 74; green, 144; blue, 226 }, ->] (110+3, 180 + 10) -- (110 + 17, 180+10);

        \draw[color={rgb, 255:red, 74; green, 144; blue, 226 }, ->] (110, 180 + 10) -- (110, 160 + 13);
        \draw[color={rgb, 255:red, 74; green, 144; blue, 226 }, ->] (110, 160 + 10) -- (110, 140+13);
        \draw[color={rgb, 255:red, 74; green, 144; blue, 226 }, ->] (110, 140 + 10) -- (110, 120+13);
        \draw[color=WildStrawberry, opacity=0.75, ->] (110, 120 + 10) -- (110, 100+13);
        \draw[color=WildStrawberry, opacity=0.75, ->] (90, 120 + 10) -- (90, 100+13);
        \draw[color=WildStrawberry, opacity=0.75, ->] (70, 100 + 10) -- (70, 80+13);
        \draw[color=WildStrawberry, opacity=0.75, ->] (50, 80 + 10) -- (50, 60+13);
        \draw[color=WildStrawberry, opacity=0.75, ->] (150, 140 + 10) -- (150, 120+13);
        \draw[color={rgb, 255:red, 74; green, 144; blue, 226 }, ->] (150, 160 + 10) -- (150, 140+13);
        \draw[color={rgb, 255:red, 74; green, 144; blue, 226 }, ->] (150, 180 + 10) -- (150, 160+13);
                \draw[color={rgb, 255:red, 74; green, 144; blue, 226 }, ->] (50, 120 + 10) -- (50, 100+13);
        \draw[color={rgb, 255:red, 74; green, 144; blue, 226 }, ->] (50, 100 + 10) -- (50, 80+13);
        \draw[color=WildStrawberry, opacity=0.75, ->] (50+3, 90 ) -- (50+17, 90);
        \draw[color={rgb, 255:red, 74; green, 144; blue, 226 }, ->] (50+3, 110 ) -- (50+17, 110);
        \draw[color=WildStrawberry, opacity=0.75, ->] (50+3, 130 ) -- (50+17, 130);
        \draw[color=WildStrawberry, opacity=0.75, ->] (70+3, 130 ) -- (70+17, 130);
        \draw[color=WildStrawberry, opacity=0.75, ->] (70+3, 110 ) -- (70+17, 110);
        \draw[color=WildStrawberry, opacity=0.75, ->] (70+3, 150 ) -- (70+17, 150);
        \draw[color=WildStrawberry, opacity=0.75, ->] (90+3, 150 ) -- (90+17, 150);
        \draw[color=WildStrawberry, opacity=0.75, ->] (90+3, 130 ) -- (90+17, 130);
        \draw[color={rgb, 255:red, 74; green, 144; blue, 226 }, ->] (90, 130+3 ) -- (90, 110+17);

               \draw[color=WildStrawberry, opacity = 0.75, ->] (130, 180 + 10) -- (130, 160+13);

               \draw[color={rgb, 255:red, 74; green, 144; blue, 226 }, ->] (70, 140 + 10) -- (70, 120+13);
                              \draw[color={rgb, 255:red, 74; green, 144; blue, 226 }, ->] (90, 140 + 10) -- (90, 120+13);

               \draw[color=WildStrawberry, opacity = 0.75, ->] (70, 120 + 10) -- (70, 100+13);

       \draw[color=WildStrawberry, opacity = 0.75, ->] (110+3, 160 + 10) -- (110 + 17, 160+10);
       \draw[color=WildStrawberry, opacity = 0.75, ->] (110+3, 140 + 10) -- (110 + 17, 140+10);
       \draw[color=WildStrawberry, opacity = 0.75, ->] (110+3, 120 + 10) -- (110 + 17, 120+10);

    \draw[color=black,line width=1.5, -](160, 140) -- (160, 200);
    \draw[color=black,line width=1.5, -](140, 140) -- (160, 140);
    \draw[color=black,line width=1.5, -](140, 140) -- (140, 180);
    \draw[color=black,line width=1.5, -](120, 180) -- (140, 180);
    \draw[color=black,line width=1.5, -](120, 180) -- (120, 120);
    \draw[color=black,line width=1.5, -](80, 120) -- (120, 120);
    \draw[color=black,line width=1.5, -](80, 100) -- (80, 120);
    \draw[color=black,line width=1.5, -](80, 100) -- (60, 100);
    \draw[color=black,line width=1.5, -](60, 80) -- (60, 100);
    \draw[color=black,line width=1.5, -](40, 80) -- (60, 80);

\fill[color=red, opacity = 0.4] (40,140) rectangle (60,200) ;
\fill[color=red, opacity = 0.4] (60,160) rectangle (80,200) ;
\fill[color=red, opacity = 0.4] (80,200) rectangle (100,180) ;
\fill[color=red, opacity = 0.4] (100,200) rectangle (120,200) ;
   \draw[color=gray,line width=1.5, -](36, 140) -- (36, 16);
   \draw[color=gray,line width=1.5, -](36, 16) -- (140, 16);
   \draw[color=gray,line width=1.5, -](100, 204) -- (224, 204);
   \draw[color=gray,line width=1.5, -](224, 100) -- (224, 204);
    \draw[black, opacity=0.5] (40,20) -- (220,20) -- (220,200) -- (40,200) -- cycle;

\fill[color=yellow, opacity = 0.4] (200,20) rectangle (220,100);
\fill[color=yellow, opacity = 0.4] (180,20) rectangle (200,80);
\fill[color=yellow, opacity = 0.4] (160,20) rectangle (180,60);
\fill[color=yellow, opacity = 0.4] (140,20) rectangle (160,40);

\draw (25,200.85) node  [font=\small,color={rgb, 255:red, 0; green, 0; blue, 0 }  ,opacity=1 ] [align=left] {\begin{minipage}[lt]{24.25pt}\setlength\topsep{0pt}
$\displaystyle ( 0,0)$
\end{minipage}};
\draw (230.17,10.85) node  [font=\small,color={rgb, 255:red, 0; green, 0; blue, 0 }  ,opacity=1 ] [align=left] {\begin{minipage}[lt]{25.61pt}\setlength\topsep{0pt}
$\displaystyle ( 1/2,1/2)$
\end{minipage}};

\draw (22,10) node  [font=\small,color=gray ,opacity=1 ] [align=left] {\begin{minipage}[lt]{24.25pt}\setlength\topsep{0pt}
$\displaystyle \mathrm{TC}_n$
\end{minipage}};

\draw (242,212) node  [font=\small,color=gray ,opacity=1 ] [align=left] {\begin{minipage}[lt]{24.25pt}\setlength\topsep{0pt}
$\displaystyle \mathrm{BC}_n$
\end{minipage}};
\end{tikzpicture}
        \caption{The connected component containing $\square_{0,0}$. Open edges are drawn in blue while closed edges are drawn in red. In black, the boundary of the connected component containing $\square_{0,0}$, as a path in the lattice with mesh size $\ell_n$. The bold-face grey lines border the regions $\mathrm{TC}_n$ and $\mathrm{BC}_n$. \\ }\label{fig:bound} 
    \end{SCfigure}
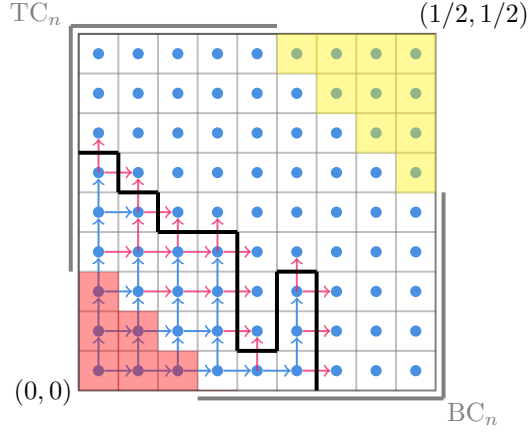
    
     Letting $P^{(i,j)}_n$ denote the number of paths started at $(i\ell_n,j\ell_n)\in \mathrm{TC}_n$ and ending in a point in $\mathrm{BC}$ it follows that 
    \begin{equation}\label{eq:goal}
        \mathbb{P}\left\{\bigcup_{i+j = \frac{n^\varepsilon}{4}}\bigcup_{k+l = 2\bnum - \frac{n^\varepsilon}{4}}\{\square_{i,j}\rightarrow \square_{k,l}\}~\Bigg |~ \mathcal{E}_n\right\} \ge 1 -\sum_{(i\ell_n,j\ell_n)\in \mathrm{TC}_n}\mathbb{P}\{P^{(i,j)}_n \ge 1\}.
    \end{equation}
    To prove Lemma~\ref{lem:perc_btw_diag} it remains to show that the right-hand side of the above inequality can be bounded from below by $1 - \exp(-\Omega(n^{\e}))$. To this end, let $(i\ell_n,j\ell_n)\in \mathrm{TC}_n$. We use the first moment method to bound $\mathbb{P}\{P^{(i,j)}_n \ge 1\}$.
    Each path that could form a boundary for the directed percolation has length $m\in \{n^{\e}/2,\dots, 2\bnum(\bnum +1) \}$, where we note that the upper bound is the total number of edges in the lattice and $2\bnum(\bnum+1) = \frac{4\sqrt{2}}{C}n^{1/3} + \frac{16}{C^2}n^{2/3} $. Lastly, in order for such a path to represent a boundary of the connected component, at least one half of its edges must intersect closed edges of $\mathbb{L}_n$, see again Figure \ref{fig:bound}. By \eqref{def:perc_prob} and \eqref{eq:ranges}, the directed edges intersecting this path are each closed independently with probability $1 - \p{B^{(n)} \ge 1}$, where $B^{(n)}$ is a  $\mathrm{Binomial}((1-t)N, \alpha|T_{n^\varepsilon/4} |)$ distributed random variable where \[|T_{n^{\e}/4}|=\frac{4(n^{\e} -2)n^{1/3 - \e}}{n^{1/3}(8\bnum - n^{\e})}.\] Since $\ell_n =(4\sqrt{2})^{-1}C n^{-1/3}$, $\bnum=(2\ell_n)^{-1}-1$ and $N= \frac{C^2}{32}n^{1/3}$, it follows that,    \begin{align*}
        \mathbb{E}\left\{B^{(n)}\right\} &= \alpha (1-t)N\frac{4(n^\varepsilon - 2)n^{1/3 - \varepsilon}}{n^{1/3}(8\bnum - n^{\varepsilon})}=\alpha(1-t)N\frac{4(n^\varepsilon - 2)n^{-\varepsilon}}{8(2\sqrt{2}n^{1/3}/C - 1) - n^{\varepsilon}},\end{align*} which tends to $\alpha(1-t)C^3/{(128\sqrt{2})}$ as  $n\rightarrow \infty$.
        
        It follows that for all $\delta > 0$, and $n$ sufficiently large, 
    \[\p{B^{(n)} \ge 1} \ge \p{\mathrm{Poisson}\left(\frac{\alpha(1-t)C^3}{128\sqrt{2}}\right) \ge 1} - \delta,\]
    and so
    \begin{align}
        1 - \p{B^{(n)} \ge 1} \leq \exp\left(-\frac{\alpha(1-t)C^3}{128\sqrt{2}}\right) + \delta.
    \end{align}
    Let $\delta > 0$. Then, we obtain that for sufficiently large $n$,
    \begin{align}
          \mathbb{E}\left\{P^{(i,j)}_n\right\} & \leq \sum_{m=n^{\e}/2}^{2\bnum(\bnum + 1)}3^m(  1 - \p{B \ge 1})^{\lfloor m/2\rfloor} \\
          & \leq \sum_{m=n^{\e}/2}^{4\bnum^{2}}3^m\left(\exp\left(\frac{\alpha(1-t)C^3}{128\sqrt{2}}\right) + \delta\right)^{\lfloor m/2\rfloor}\notag \\
          & \leq \ell_n^{-2}\left(10 \exp\left(-\frac{\alpha(1-t)C^3}{128\sqrt{2}}\right)\right)^{n^{\e}/4},
    \end{align}
    where the final inequality follows from taking $\delta>0$ sufficiently small. Taking $C>0$ sufficiently large, so that 
    \begin{align}
        10 \exp\left(-\frac{\alpha(1-t)C^3}{128\sqrt{2}}\right)  < e^{-4},
    \end{align}
    we obtain that $ \mathbb{E}\{P^{(i,j)}_n\} =\exp(-\Omega(n^{-\e}))$. (In fact, taking $t=1/100$, $C\ge 8\alpha^{-1/3}$ suffices).  The result then follows from (\ref{eq:goal}) by an application of Markov's inequality and the fact that $\#\{(i\ell_n, j\ell_n) \in \mathrm{TC}_n\} = O(\ell_n^{-1}) = O(n^{1/3})$.\qedhere
\end{proof}

\subsection{Proof of Proposition \ref{prop:lower_bound}.}
In this subsection we will show how Lemmas~\ref{lem:perc_to_diag} and ~\ref{lem:perc_btw_diag} imply {that for $\e \in (0, 1/9)$,}
\begin{align}\label{eq:prob_zero_to_one}
    \p{\tpathn{(0,0)}{(\frac{1}{2},\frac{1}{2})}} & \geq 1 -\exp(-\Omega(n^{1/3-{3}\e})),
\end{align}
for $(\mathcal{G}_n)_{n\geq 1}$ as in Proposition~\ref{prop:lower_bound} and where we work with a modified point process obtained by adding the points $(0,0)$ and $(1/2,1/2)$ to $\X_n$. To see why Proposition~\ref{prop:lower_bound} follows from \eqref{eq:prob_zero_to_one}, note that the maximal distance between two vertices in $\G_n$ is $\sqrt{2}/2$, for instance between $(0,0)$ and $(1/2,1/2)$. Since we are working on the torus $[0,1]^d$, the same proof construction can be used for any pairs of points in $\G_n$, by shifting the first and third regions (highlighted in red and yellow respectively in Figure \ref{fig:shift}), and shrinking the second region, (the white region in Figure \ref{fig:shift}) appropriately while keeping the box side-length $\ell_n = (4\sqrt{2})^{-1}r_n$ as before. If the first and third regions overlap (and so there is no white region in Figure \ref{fig:shift}), we shrink the first and third regions appropriately while keeping the same box side-length as in the original construction. See Figure \ref{fig:shift}.  

\begin{figure}[h]
\centering
\begin{minipage}{0.5\textwidth}
\resizebox{0.9\textwidth}{!}{\tikzset{every picture/.style={line width=0.75pt}} 

\begin{tikzpicture}[x=0.75pt,y=0.75pt,yscale=-1,xscale=1]

\draw  [draw opacity=0] (40,20) -- (340,20) -- (340,320) -- (40,320) -- cycle ; \draw  [color={rgb, 255:red, 155; green, 155; blue, 155 }  ,draw opacity=0.5 ] (40,20) -- (40,320)(60,20) -- (60,320)(80,20) -- (80,320)(100,20) -- (100,320)(120,20) -- (120,320)(140,20) -- (140,320)(160,20) -- (160,320)(180,20) -- (180,320)(200,20) -- (200,320)(220,20) -- (220,320)(240,20) -- (240,320)(260,20) -- (260,320)(280,20) -- (280,320)(300,20) -- (300,320)(320,20) -- (320,320) ; \draw  [color={rgb, 255:red, 155; green, 155; blue, 155 }  ,draw opacity=0.5 ] (40,20) -- (340,20)(40,40) -- (340,40)(40,60) -- (340,60)(40,80) -- (340,80)(40,100) -- (340,100)(40,120) -- (340,120)(40,140) -- (340,140)(40,160) -- (340,160)(40,180) -- (340,180)(40,200) -- (340,200)(40,220) -- (340,220)(40,240) -- (340,240)(40,260) -- (340,260)(40,280) -- (340,280)(40,300) -- (340,300) ; \draw  [color={rgb, 255:red, 155; green, 155; blue, 155 }  ,draw opacity=0.5 ]  ;
\draw   (40,20) -- (340,20) -- (340,320) -- (40,320) -- cycle ;
\draw [line width = 0.5mm,color={rgb, 255:red, 74; green, 144; blue, 226 }  ,draw opacity=1 ]   (70,250) .. controls (70.33,52) and (291,199.33) .. (290,70) ;
\draw [line width = 0.5mm,color={rgb, 255:red, 74; green, 144; blue, 226 }  ,draw opacity=1 ]   (50,310) -- (130,310) ;
\draw [line width = 0.5mm,color={rgb, 255:red, 74; green, 144; blue, 226 }  ,draw opacity=1 ]   (50,310) -- (50,230) ;
\draw [line width = 0.5mm,color={rgb, 255:red, 74; green, 144; blue, 226 }  ,draw opacity=1 ]   (50,290) -- (110,290) ;
\draw [line width = 0.5mm,color={rgb, 255:red, 74; green, 144; blue, 226 }  ,draw opacity=1 ]   (50,270) -- (90,270) ;
\draw [line width = 0.5mm,color={rgb, 255:red, 74; green, 144; blue, 226 }  ,draw opacity=1 ]   (50,250) -- (70,250) ;
\draw [line width = 0.5mm,color={rgb, 255:red, 74; green, 144; blue, 226 }  ,draw opacity=1 ]   (70,310) -- (70,250) ;
\draw [line width = 0.5mm,color={rgb, 255:red, 74; green, 144; blue, 226 }  ,draw opacity=1 ]   (90,310) -- (90,270) ;
\draw [line width = 0.5mm,color={rgb, 255:red, 74; green, 144; blue, 226 }  ,draw opacity=1 ]   (110,310) -- (110,290) ;
\draw [line width = 0.5mm,color={rgb, 255:red, 74; green, 144; blue, 226 }  ,draw opacity=1 ]   (330,30) -- (250,30) ;
\draw [line width = 0.5mm,color={rgb, 255:red, 74; green, 144; blue, 226 }  ,draw opacity=1 ]   (330,30) -- (330,110) ;
\draw [line width = 0.5mm,color={rgb, 255:red, 74; green, 144; blue, 226 }  ,draw opacity=1 ]   (330,50) -- (270,50) ;
\draw [line width = 0.5mm,color={rgb, 255:red, 74; green, 144; blue, 226 }  ,draw opacity=1 ]   (330,70) -- (290,70) ;
\draw [line width = 0.5mm,color={rgb, 255:red, 74; green, 144; blue, 226 }  ,draw opacity=1 ]   (330,90) -- (310,90) ;
\draw [line width = 0.5mm, color={rgb, 255:red, 74; green, 144; blue, 226 }  ,draw opacity=1]   (310,30) -- (310,90) ;
\draw [line width = 0.5mm,color={rgb, 255:red, 74; green, 144; blue, 226 }  ,draw opacity=1 ]   (290,30) -- (290,70) ;
\draw [line width = 0.5mm,color={rgb, 255:red, 74; green, 144; blue, 226 }  ,draw opacity=1 ]   (270,30) -- (270,50) ;

\draw (39.83,327.85) node  [font=\small,color={rgb, 255:red, 0; green, 0; blue, 0 }  ,opacity=1 ] [align=left] {\begin{minipage}[lt]{24.25pt}\setlength\topsep{0pt}
$\displaystyle ( 0,0)$
\end{minipage}};
\draw (340.17,10.85) node  [font=\small,color={rgb, 255:red, 0; green, 0; blue, 0 }  ,opacity=1 ] [align=left] {\begin{minipage}[lt]{25.61pt}\setlength\topsep{0pt}
$\displaystyle ( 1/2,1/2)$
\end{minipage}};
\draw (110,329) node  [font=\small,color={rgb, 255:red, 0; green, 0; blue, 0 }  ,opacity=1 ] [align=left] {\begin{minipage}[lt]{24.25pt}\setlength\topsep{0pt}
$\displaystyle (n^\varepsilon \ell_n/4,0)$
\end{minipage}};
\draw (250,10) node  [font=\small,color={rgb, 255:red, 0; green, 0; blue, 0 }  ,opacity=1 ] [align=right] {\begin{minipage}[lt]{140pt}\setlength\topsep{0pt}
$\displaystyle (1/2 - \ell_n(1+n^\varepsilon/4), 1/2)$
\end{minipage}};
\fill[color=red, opacity = 0.2] (40,240) rectangle (60,320) ;
\fill[color=red, opacity = 0.2] (60,260) rectangle (80,320) ;
\fill[color=red, opacity = 0.2] (80,280) rectangle (100,320) ;
\fill[color=red, opacity = 0.2] (100,300) rectangle (120,320) ;
\fill[color=red, opacity = 0.2] (120,320) rectangle (140,320) ;
\fill[color=yellow, opacity = 0.2] (240,20) rectangle (260,40) ;
\fill[color=yellow, opacity = 0.2] (260,20) rectangle (280,60) ;
\fill[color=yellow, opacity = 0.2] (280,20) rectangle (300,80) ;
\fill[color=yellow, opacity = 0.2] (300,20) rectangle (320,100) ;
\fill[color=yellow, opacity = 0.2] (320,20) rectangle (340,120) ;

\end{tikzpicture}}
\end{minipage}%
\begin{minipage}{0.5\textwidth}
\resizebox{0.9\textwidth}{!}{\tikzset{every picture/.style={line width=0.75pt}} 

\begin{tikzpicture}[x=0.75pt,y=0.75pt,yscale=-1,xscale=1]

    \draw  [draw opacity=0] (40,20) -- (340,20) -- (340,320) -- (40,320) -- cycle ; 
    
\draw  [draw opacity=1] 
(40, 320) -- (40, 80);
\draw  [draw opacity=1] 
(40, 80) -- (280, 80);
\draw  [draw opacity=1] 
(280, 80) -- (280, 320);
\draw  [draw opacity=1] 
(280, 320) -- (40, 320);
\draw  [color={rgb, 255:red, 155; green, 155; blue, 155 }  ,draw opacity=0.5 ] (60, 320) -- (60,80);
\draw  [color={rgb, 255:red, 155; green, 155; blue, 155 }  ,draw opacity=0.5 ] (80, 320) -- (80,80);
\draw  [color={rgb, 255:red, 155; green, 155; blue, 155 }  ,draw opacity=0.5 ] (100, 320) -- (100,80);
\draw  [color={rgb, 255:red, 155; green, 155; blue, 155 }  ,draw opacity=0.5 ] (120, 320) -- (120,80);
\draw  [color={rgb, 255:red, 155; green, 155; blue, 155 }  ,draw opacity=0.5 ] (140, 320) -- (140,80);
\draw  [color={rgb, 255:red, 155; green, 155; blue, 155 }  ,draw opacity=0.5 ] (160, 320) -- (160,80);
\draw  [color={rgb, 255:red, 155; green, 155; blue, 155 }  ,draw opacity=0.5 ] (180, 320) -- (180,80);
\draw  [color={rgb, 255:red, 155; green, 155; blue, 155 }  ,draw opacity=0.5 ] (200, 320) -- (200,80);
\draw  [color={rgb, 255:red, 155; green, 155; blue, 155 }  ,draw opacity=0.5 ] (220, 320) -- (220,80);
\draw  [color={rgb, 255:red, 155; green, 155; blue, 155 }  ,draw opacity=0.5 ] (240, 320) -- (240,80);
\draw  [color={rgb, 255:red, 155; green, 155; blue, 155 }  ,draw opacity=0.5 ] (260, 320) -- (260,80);
\draw  [color={rgb, 255:red, 155; green, 155; blue, 155 }  ,draw opacity=0.5 ] (40, 280) -- (280,280);
\draw  [color={rgb, 255:red, 155; green, 155; blue, 155 }  ,draw opacity=0.5 ] (40, 300) -- (280,300);
\draw  [color={rgb, 255:red, 155; green, 155; blue, 155 }  ,draw opacity=0.5 ] (40, 260) -- (280,260);\draw  [color={rgb, 255:red, 155; green, 155; blue, 155 }  ,draw opacity=0.5 ] (40, 240) -- (280,240);
\draw  [color={rgb, 255:red, 155; green, 155; blue, 155 }  ,draw opacity=0.5 ] (40, 220) -- (280,220);\draw  [color={rgb, 255:red, 155; green, 155; blue, 155 }  ,draw opacity=0.5 ] (40, 200) -- (280,200);\draw  [color={rgb, 255:red, 155; green, 155; blue, 155 }  ,draw opacity=0.5 ] (40, 180) -- (280,180);
\draw  [color={rgb, 255:red, 155; green, 155; blue, 155 }  ,draw opacity=0.5 ] (40, 160) -- (280,160);
\draw  [color={rgb, 255:red, 155; green, 155; blue, 155 }  ,draw opacity=0.5 ] (40, 140) -- (280,140);\draw  [color={rgb, 255:red, 155; green, 155; blue, 155 }  ,draw opacity=0.5 ] (40, 120) -- (280,120);\draw  [color={rgb, 255:red, 155; green, 155; blue, 155 }  ,draw opacity=0.5 ] (40, 100) -- (280,100);
\draw   (40,20) -- (340,20) -- (340,320) -- (40,320) -- cycle ;
\draw [line width = 0.5mm,color={rgb, 255:red, 74; green, 144; blue, 226 }  ,draw opacity=1 ] (50, 230) -- (50, 310);
\draw [line width = 0.5mm,color={rgb, 255:red, 74; green, 144; blue, 226 }  ,draw opacity=1 ] (70, 250) -- (70, 310);
\draw [line width = 0.5mm,color={rgb, 255:red, 74; green, 144; blue, 226 }  ,draw opacity=1 ] (90, 270) -- (90, 310);
\draw [line width = 0.5mm,color={rgb, 255:red, 74; green, 144; blue, 226 }  ,draw opacity=1 ] (110, 290) -- (110, 310);

\draw [line width = 0.5mm,color={rgb, 255:red, 74; green, 144; blue, 226 }  ,draw opacity=1 ] (50, 310) -- (130, 310);
\draw [line width = 0.5mm,color={rgb, 255:red, 74; green, 144; blue, 226 }  ,draw opacity=1 ] (50, 290) -- (110, 290);
\draw [line width = 0.5mm,color={rgb, 255:red, 74; green, 144; blue, 226 }  ,draw opacity=1 ] (50, 270) -- (90, 270);
\draw [line width = 0.5mm,color={rgb, 255:red, 74; green, 144; blue, 226 }  ,draw opacity=1 ] (50, 250) -- (70, 250);
\draw [line width = 0.5mm,color={rgb, 255:red, 74; green, 144; blue, 226 }  ,draw opacity=1 ]   (70,250) .. controls (80,80) and (260,300) .. (230,130) ;

\draw [line width = 0.5mm,color={rgb, 255:red, 74; green, 144; blue, 226 }  ,draw opacity=1 ] (210, 90) -- (210, 110);
\draw [line width = 0.5mm,color={rgb, 255:red, 74; green, 144; blue, 226 }  ,draw opacity=1 ] (230, 90) -- (230, 130);
\draw [line width = 0.5mm,color={rgb, 255:red, 74; green, 144; blue, 226 }  ,draw opacity=1 ] (250, 90) -- (250, 150);
\draw [line width = 0.5mm,color={rgb, 255:red, 74; green, 144; blue, 226 }  ,draw opacity=1 ] (270, 90) -- (270, 170);
\draw [line width = 0.5mm,color={rgb, 255:red, 74; green, 144; blue, 226 }  ,draw opacity=1 ] (190, 90) -- (270, 90);
\draw [line width = 0.5mm,color={rgb, 255:red, 74; green, 144; blue, 226 }  ,draw opacity=1 ] (210, 110) -- (270, 110);
\draw [line width = 0.5mm,color={rgb, 255:red, 74; green, 144; blue, 226 }  ,draw opacity=1 ] (230, 130) -- (270, 130);
\draw [line width = 0.5mm,color={rgb, 255:red, 74; green, 144; blue, 226 }  ,draw opacity=1 ] (250, 150) -- (270, 150);

\fill[color=red, opacity = 0.2] (40,240) rectangle (60,320) ;
\fill[color=red, opacity = 0.2] (60,260) rectangle (80,320) ;
\fill[color=red, opacity = 0.2] (80,280) rectangle (100,320) ;
\fill[color=red, opacity = 0.2] (100,300) rectangle (120,320) ;
\fill[color=red, opacity = 0.2] (120,320) rectangle (140,320) ;

\fill[color=yellow, opacity = 0.2] (180,80) rectangle (200,100) ;
\fill[color=yellow, opacity = 0.2] (200,80) rectangle (220,120) ;
\fill[color=yellow, opacity = 0.2] (220,80) rectangle (240,140) ;
\fill[color=yellow, opacity = 0.2] (240,80) rectangle (260,160) ;
\fill[color=yellow, opacity = 0.2] (260,80) rectangle (280,180) ;
\draw (39.83,327.85) node  [font=\small,color={rgb, 255:red, 0; green, 0; blue, 0 }  ,opacity=1 ] [align=left] {\begin{minipage}[lt]{24.25pt}\setlength\topsep{0pt}
$\displaystyle ( 0,0)$
\end{minipage}};
\draw (340.17,10.85) node  [font=\small,color={rgb, 255:red, 0; green, 0; blue, 0 }  ,opacity=1 ] [align=left] {\begin{minipage}[lt]{25.61pt}\setlength\topsep{0pt}
$\displaystyle ( 1/2,1/2)$
\end{minipage}};

\draw (280,70) node  [font=\small,color={rgb, 255:red, 0; green, 0; blue, 0 }  ,opacity=1 ] [align=left] {\begin{minipage}[lt]{25.61pt}\setlength\topsep{0pt}
$\displaystyle (x,y)$
\end{minipage}};

\draw (200,70) node  [font=\small,color={rgb, 255:red, 0; green, 0; blue, 0 }  ,opacity=1 ] [align=left] {\begin{minipage}[lt]{100pt}\setlength\topsep{0pt}
$\displaystyle (x - \ell_n(1+n^\varepsilon/4),y)$
\end{minipage}};
\draw (170,330) node  [font=\small,color={rgb, 255:red, 0; green, 0; blue, 0 }  ,opacity=1 ] [align=left] {\begin{minipage}[lt]{100pt}\setlength\topsep{0pt}
$\displaystyle (n^\varepsilon \ell_n/4,0)$
\end{minipage}};

\end{tikzpicture}}
\end{minipage}
\caption{ On the left, an illustration of the construction of $\protect \tpathn{(0,0)}{(\frac{1}{2},\frac{1}{2})}$. On the right, an illustration of the alteration of this construction for a temporal path between $(0,0)$ and an arbitrary point $(x,y)$.
}\label{fig:shift}
\end{figure}
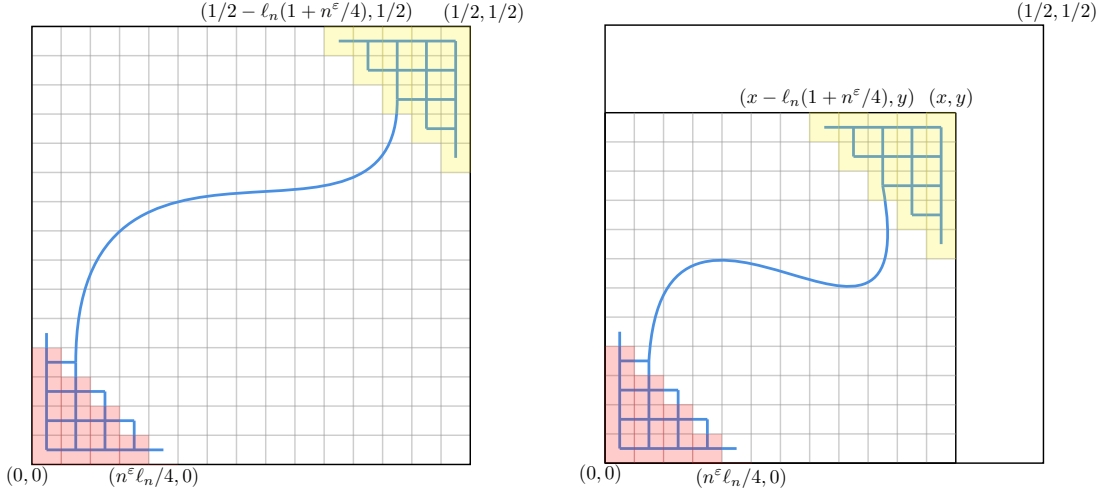

A union bound over all possible pairs of vertices then gives
\begin{align}
    \p{\forall u,v \in \G_n, \tpathn{u}{v}} & \geq 1 - n^2 \exp(-\Omega(n^{1/3-{3}\e})),
\end{align}
proving Proposition~\ref{prop:lower_bound} for $d=2$.
\begin{proof}
    [Proof of Proposition \ref{prop:lower_bound} with $d = 2$]
 As discussed above, to prove Proposition \ref{prop:lower_bound}, it suffices to prove \eqref{eq:prob_zero_to_one}. Recall $(\square_{i,j})_{i,j}$ from \eqref{def:boxes}. We begin by defining $A_n$ as the event that the number of vertices in each box $\square_{i,j},$ is well-concentrated. More precisely, for $t\in (0,1)$ let 
\begin{align}\label{def:event_concentration_vij}
    A_n = A_n(t) := \left\{\max_{i,j}\left||\X_n \cap \square_{i,j}| - N\right| <tN\right\},
\end{align}
recalling that $|\X_n \cap \square_{i,j}|$ is a $\operatorname{Binomial}(n,\text{vol}(\square_{0,0}))$ random variable and  $N = n \text{vol}(\square_{0,0}) = (C^2/32)n^{1/3}$.  By the union bound and standard concentration inequalities for Binomial random variables \cite[Theorem~2.1]{janson2000random}, it holds that for all $t\in (0,1)$, for $n \ge 1$,
   \begin{equation}\label{lem:concentration_vij}\mathbf{P}\left\{ A_n(t)^c\right\} \le \frac{16n^{2/3}}{C^2}\exp \left( - \frac{(Ct)^2}{8}n^{1/3}\right){= \exp(-\Omega(n^{1/3}))}.\end{equation}
    We therefore focus on proving that \[\p{\tpathn{(0,0)}{(\frac{1}{2},\frac{1}{2})}\mid A_n}\] can be bounded from below by the right-hand side of \eqref{eq:prob_zero_to_one} for the remainder of this section. We lower bound this probability by that of having a temporal path from $(0,0)$ to $(\frac{1}{2},\frac{1}{2})$ using only edges $xy$ satisfying $x\in \X_n \cap \square_{i,j}$, $y \in \X_n \cap (\square_{i+1,j} \cup \square_{i,j+1})$ and $\tau_{xy}\in T_{i,j}$. Observe that by our choice of box side-length $\ell_n = (4\sqrt{2})^{-1}r_n$ we have that an edge $xy$ with $x\in \X_n \cap \square_{i,j}$, $y \in \X_n \cap (\square_{i+1,j} \cup \square_{i,j+1})$ exists with probability $\alpha$. Further remark that for all $i,j \in [\bnum]_0$ and $\vec{e} = (\square_{i,j}, \square_{i',j'}) \in \mathbb{E}_n$,
   \begin{align}
       & \p{\forall x\in \X_n \cap \square_{i,j}, \, \exists y \in \X_n \cap \square_{i',j'} \text{ s.t. } xy \in E(G_n), \tau_{xy} \in T_{i+j} {\mid} A_n}\notag \\
       & \geq  \p{\prod_{k\in [(1+t)N]}B_{\vec{e},k} \geq 1 },
   \end{align}
   and given $x\in \X_n \cap \square_{i,j} $,
   \begin{align}
       \p{\exists y \in \X_n \cap \square_{i',j'} \text{ s.t. } xy \in E(G_n), \tau_{xy} \in T_{i+j} {\mid} A_n} \geq  \p{B_{\vec{e},1} \geq 1 },
   \end{align}
   where $(B_{\vec{e}, k})_{k\geq 1}$ are i.i.d.\ $\text{Binomial}((1-t)N, \alpha|T_{i,j}|)$ distributed random variables. 
   {Finally, we will ensure that if we can reach some vertex in the final box with a temporal path using only time-stamps in $[0,1-n^{-\varepsilon}]$, then we can complete the temporal path to $(\tfrac{1}{2},\tfrac{1}{2})$ using two further edges. To be precise, }{on the event $A_n$, given $x\in \X_n \cap \square_{\bnum, \bnum}$, the probability of having a temporal path from $x$ to $(\frac{1}{2},\frac{1}{2})$ only using edges with time-stamps in $(1-n^{-\varepsilon}, 1]$ is bounded from below by the probability of having $n^{1/3-2\varepsilon}$ vertices $y\in \X_n \cap \square_{\bnum, \bnum}$ such that $xy \in E(G_n)$ and $\tau_{xy}\in (1-n^{-\varepsilon}, 1 - n^{-\varepsilon}/2]$ and such that one of these $y$ satisfies $y(\frac{1}{2},\frac{1}{2}) \in E(G_n)$ and $\tau_{y(\frac{1}{2},\frac{1}{2})}\in (1 - n^{-\varepsilon}/2, 1]$. This probability can further be bounded from below by 
   \begin{align}\label{eq:last_step_lowerbd}
       \p{\text{Binomial}((1-t)N, \alpha n^{-\varepsilon}/2)\geq n^{1/3-2\varepsilon}} \times  \p{\text{Binomial}(n^{1/3-2\varepsilon}, \alpha n^{-\varepsilon}/2)\geq 1}. 
   \end{align}
   The mean of the Binomial in the first probability is of the order of $n^{1/3 - \varepsilon}$, thus by standard concentration inequalities for Binomial random variables \cite[Theorem~2.1]{janson2000random}, this probability is at least $1- \exp(-\Omega(n^{1/3-\varepsilon}))$. The second term in \eqref{eq:last_step_lowerbd} equals $1-(1- \alpha n^{-\varepsilon}/2)^{n^{1/3 - 2\varepsilon}} {=} 1 - \exp(-\Omega(n^{1/3 - 3\varepsilon}))$. Putting everything together, on the event $A_n$, given $x\in \X_n \cap \square_{\bnum, \bnum}$, the probability of having a temporal path from $x$ to $(\frac{1}{2},\frac{1}{2})$ only using edges with time-stamps in $(1-n^{-\varepsilon}, 1]$ is bounded from below by $1-\exp(-\Omega(n^{1/3 - 3\varepsilon}))$.}\\
    From this, we can see that
   \begin{align}\label{eq:prob_zero_to_half}
       \p{\tpathn{(0,0)}{(\frac{1}{2},\frac{1}{2})} \mid A_n}& \ge \left(1 - \exp(- \Omega(n^{1/3 -3\e}))\right)\mathbb{P}\left\{\square_{0,0}\rightarrow \square_{\bnum,\bnum}\right\}.
   \end{align}
    By independence of the $(p_{\vec{e}})_{\vec{e}}$,  $\mathbb{P}\{\square_{0,0} \to \square_{\bnum,\bnum}\}$ can be bounded from below by 
   \begin{align}
      & \P\left\{ \bigcap_{i+j = \frac{n^{\varepsilon}}{4}-1} \{\square_{0,0} \to \square_{i,j}\}\right\} \label{eq:prob_to_diag}  \\
      & \quad\times\P\left\{ \bigcup_{i+j = \frac{n^{\varepsilon}}{4}-1} \bigcup_{k+l = 2\bnum - n^{\e}/4} \{\square_{i,j} \to \square_{k,l}\}\right\} \label{eq:prob_btw_diag}\\
         &\quad\times \P\left\{ \bigcap_{i + j = 2\bnum - n^{\e}/4 }\{\square_{i,j} \to \square_{\bnum,\bnum} \}\right\}. \label{eq:prob_from_diag}
   \end{align}
    Lemmas~\ref{lem:perc_to_diag} and \ref{lem:perc_btw_diag} yield a lower bound for \eqref{eq:prob_to_diag} and \eqref{eq:prob_btw_diag}. Further, since the time-stamp ranges in the lower left and upper right triangle regions, $\{\square_{i,j}: i+ j \leq n^{\e}/4-1\}$ and $\{\square_{i,j}: i+ j \geq 2\bnum - n^{\e}/4\}$ respectively, all satisfy $|T_{i+j}| = n^{-\e}$, we can argue by symmetry that the lower bound in Lemma~\ref{lem:perc_to_diag} also holds for \eqref{eq:prob_from_diag}, giving
    \begin{equation*}
         \P\{\square_{0,0} \to \square_{\bnum, \bnum}\}  \geq 1 - \exp(-\Omega(n^{1/3-\e})) - \exp(-\Omega(n^{\e})).
  \end{equation*}
    {Together with \eqref{eq:prob_zero_to_half} and \eqref{lem:concentration_vij} we conclude \eqref{eq:prob_zero_to_one}, proving Proposition \ref{prop:lower_bound}.}  \qedhere
   \end{proof}
    
   \subsubsection{Extension to dimension $d > 2$}\label{sec:dim}
    The proof of Proposition \ref{prop:lower_bound} in dimension $d > 2$ is almost identical to that in dimension $d = 2$. Let $\mbox{\mancube}_{\vec{0}}, \mbox{\mancube}_{\vec{1}}$ be the $d$-dimensional boxes in the dual grid lattice on $[0,1]^d$ with mesh size $\ell_n = {(4\sqrt{d})^{-1}}r_n$ which contain the points $\vec{0}$, $\vec{1}\in [0,1]^d$ respectively. By considering the same percolation model as in subsection \ref{sec:directed_percolation} along the $d$-dimensional boxes that intersect a slab which contains the points $\mbox{\mancube}_{\vec{0}},$ and $ \mbox{\mancube}_{\vec{1}}$, $d$-dimensional versions of Lemmas \ref{lem:perc_to_diag} and \ref{lem:perc_btw_diag} hold using the same proof ideas. This implies that there exists an open path from $\mbox{\mancube}_{\vec{0}}$ to $\mbox{\mancube}_{\vec{1}}$. The remainder of the proof follows similarly to that of Proposition \ref{prop:lower_bound} with $d = 2$. 

\section*{Acknowledgements}

This research was initiated during the Seventeenth Annual Workshop on Probability and Combinatorics at McGill University's Bellairs Institute in Holetown, Barbados. We thank the Bellairs Institute for its hospitality. 
AB was supported by NSF GRFP 2141064 and Simons Investigator Award 622132 to Elchanan Mossel. 
SD acknowledges
the financial support of the CogniGron research center
and the Ubbo Emmius Funds (University of Groningen).
CK was supported by DFG project 444092244 “Condensation in random geometric graphs” within the priority programme SPP 2265.
GL acknowledges the support of Ayudas Fundación BBVA a
Proyectos de Investigación Científica 2021 and
the Spanish Ministry of Economy and Competitiveness grant PID2022-138268NB-I00, financed by MCIN/AEI/10.13039/501100011033,
FSE+MTM2015-67304-P, and FEDER, EU.).
RM was supported by the EPSRC Centre for Doctoral Training in Mathematics of Random Systems: Analysis, Modelling and Simulation (EP/S023925/1).

\bibliographystyle{abbrv}
\bibliography{biblio}

\begin{thebibliography}{10}

\bibitem{Angel_Ferber_Sudakov_Tassion_2020}
O.~Angel, A.~Ferber, B.~Sudakov, and V.~Tassion.
\newblock Long monotone trails in random edge-labellings of random graphs.
\newblock {\em Combinatorics, Probability and Computing}, 29(1):22–30, 2020.

\bibitem{atamanchuk2024size}
C.~Atamanchuk, L.~Devroye, and G.~Lugosi.
\newblock On the size of temporal cliques in subcritical random temporal
  graphs.
\newblock {\em arXiv preprint arXiv:2404.04462}, 2024.

\bibitem{becker2023giant}
R.~Becker, A.~Casteigts, P.~Crescenzi, B.~Kodric, M.~Renken, M.~Raskin, and
  V.~Zamaraev.
\newblock {Giant Components in Random Temporal Graphs}.
\newblock In N.~Megow and A.~Smith, editors, {\em Approximation, Randomization,
  and Combinatorial Optimization. Algorithms and Techniques (APPROX/RANDOM
  2023)}, volume 275 of {\em Leibniz International Proceedings in Informatics
  (LIPIcs)}, pages 29:1--29:17, Dagstuhl, Germany, 2023. Schloss Dagstuhl --
  Leibniz-Zentrum f{\"u}r Informatik.

\bibitem{boucheron2003concentration}
S.~Boucheron, G.~Lugosi, and O.~Bousquet.
\newblock Concentration inequalities.
\newblock In {\em Summer school on machine learning}, pages 208--240. Springer,
  2003.

\bibitem{NicNinaGabor}
N.~Broutin, N.~Kamčev, and G.~Lugosi.
\newblock Increasing paths in random temporal graphs, 2023.

\bibitem{bucic2020nearly}
M.~Buci{\'c}, M.~Kwan, A.~Pokrovskiy, B.~Sudakov, T.~Tran, and A.~Z. Wagner.
\newblock Nearly-linear monotone paths in edge-ordered graphs.
\newblock {\em Israel Journal of Mathematics}, 238(2):663--685, 2020.

\bibitem{calderbank1984increasing}
A.~R. Calderbank, F.~R. Chung, and D.~G. Sturtevant.
\newblock Increasing sequences with nonzero block sums and increasing paths in
  edge-ordered graphs.
\newblock {\em Discrete mathematics}, 50:15--28, 1984.

\bibitem{casteigts2021}
A.~Casteigts, M.~Raskin, M.~Renken, and V.~Zamaraev.
\newblock Sharp thresholds in random simple temporal graphs.
\newblock In {\em 2021 IEEE 62nd Annual Symposium on Foundations of Computer
  Science (FOCS)}, pages 319--326, 2022.

\bibitem{chvatal1971some}
V.~Chv{\'a}tal and J.~Koml{\'o}s.
\newblock Some combinatorial theorems on monotonicity.
\newblock {\em Canadian Mathematical Bulletin}, 14(2):151--157, 1971.

\bibitem{graham1973increasing}
R.~L. Graham and D.~J. Kleitman.
\newblock Increasing paths in edge ordered graphs.
\newblock {\em Periodica Mathematica Hungarica}, 3(1-2):141--148, 1973.

\bibitem{janson2000random}
S.~Janson, T.~Luczak, and A.~Rucinski.
\newblock {\em Random graphs}.
\newblock Wiley-Interscience Series in Discrete Mathematics and Optimization.
  Wiley-Interscience, New York, 2000.

\bibitem{lavrov2016increasing}
M.~Lavrov and P.-S. Loh.
\newblock Increasing {Hamiltonian} paths in random edge orderings.
\newblock {\em Random Structures \& Algorithms}, 48(3):588--611, 2016.

\bibitem{martinsson2019most}
A.~Martinsson.
\newblock Most edge-orderings of {$K_n$} have maximal altitude.
\newblock {\em Random Structures \& Algorithms}, 54(3):559--585, 2019.

\bibitem{Penrose2013}
M.~Penrose.
\newblock Connectivity of soft random geometric graphs.
\newblock {\em The Annals of Applied Probability}, 26, 11 2013.

\bibitem{Tang2013ApplicationsOT}
J.~Tang, I.~Leontiadis, S.~Scellato, V.~Nicosia, C.~Mascolo, M.~Musolesi, and
  V.~Latora.
\newblock Applications of temporal graph metrics to real-world networks.
\newblock {\em Temporal Networks}, pages 135--159, 2013.

\end{thebibliography}

\appendix

\end{document}